\documentclass[11pt,reqno]{amsart}
\usepackage{hyperref}
\usepackage{bm,amsfonts,mathrsfs,bbm,latexsym,rawfonts,amsmath,amssymb,amsthm,epsfig}
\usepackage{fullpage, setspace, color}
\usepackage{mathrsfs}
\allowdisplaybreaks
\usepackage{todonotes}
\usepackage[numbers,sort&compress]{natbib}

\newtheorem{theorem}{Theorem}[section]
\newtheorem{lemma}[theorem]{Lemma}

\newtheorem{definition}[theorem]{Definition}
\newtheorem{remark}[theorem]{Remark}
\newtheorem{proposition}[theorem]{Proposition}
\numberwithin{equation}{section}

\newcommand{\al}{\alpha}

\newcommand{\De}{\Delta}
\newcommand{\ep}{\varepsilon}

\newcommand{\Om}{\Omega}
\newcommand{\ga}{\gamma}

\renewcommand{\P}{\mathcal{P}}

\def\<{\left\langle} \def\>{\right\rangle}

\newcommand{\n}{\nabla}

\newcommand{\p}{\partial}

\hypersetup{
  colorlinks=true,
  linkcolor=blue,
  citecolor=blue,
  filecolor=blue,
  urlcolor=blue
}

\graphicspath{{./figures/}}

\newcommand{\norm}[1]{\lVert#1\rVert}



\begin{document}
\begin{sloppypar}
    
\title{Global existence of weak solutions for Landau-Lifshitz equation with helical derivatives}
\author{Bo Chen}
\address{School of Mathematics, South China University of Technology, Guangzhou, 510640, People's Republic of China}
\email{cbmath@scut.edu.cn}

\author{ Zhen Qiu}
\address{School of Mathematics and Information Science, Guangzhou University}
\email{qiuzhen97@foxmail.com}

\begin{abstract}
In this paper, we investigate the chiral boundary value problem for the Landau-Lifshitz equation with helical derivatives. By introducing Sobolev spaces adapted to the helical derivative and establishing energy estimates that are compatible with the chiral boundary condition, we prove the global existence of weak solutions to this problem, both in the presence and in the absence of damping terms.
\end{abstract}

\maketitle

\section{Introduction}

\subsection{Micromagnetic energy with Dzyaloshinskii-Moriya interaction (DMI)}
In physics, Dzyaloshinsky \cite{DZYALOSHINSKY1958241} and Moriya \cite{PhysRev.120.91} independently established the microscopic theory of antisymmetric exchange interactions in magnetic systems lacking inversion symmetry. Dzyaloshinsky's thermodynamic theory \cite{DZYALOSHINSKY1958241} explained the emergence of ``weak'' ferromagnetism in antiferromagnetic materials, such as $\alpha$-Fe$_2$O$_3$, attributing it to relativistic spin-lattice interactions and the magnetic dipole interactions within the crystal lattice. Meanwhile, Moriya \cite{PhysRev.120.91} provided a quantum-mechanical formulation based on spin-orbit coupling, demonstrating that antisymmetric exchange arises from superexchange processes in non-centrosymmetric crystals, where inversion symmetry is broken. This interaction, now known as the Dzyaloshinskii-Moriya interaction (DMI), plays a crucial role in stabilizing chiral spin textures and provides the theoretical foundation for the description of chiral magnetic skyrmions. The magnetic skyrmions,  vortex-like configurations, have been observed in several materials such as MnSi, FeGe, and various multilayer systems. 

Mathematically, the DMI is explained by the DMI energy density, which can be expressed as 
$$\bm{m} \cdot (\nabla \times \bm{m}),$$
where $\bm{m}$ is the normalized magnetization defined in a domain in $\mathbb{R}^3$. This term breaks chiral symmetry and favors non-collinear spin arrangements, such as helices and skyrmions. In continuum models, the full micromagnetic energy often takes the following form.

Let $\Omega$ be a bounded domain in $\mathbb{R}^3$, and $\mathbb{S}^2$ denote the unit sphere in $\mathbb{R}^3$. For a magnetization field $\bm{m} :\Omega\times [0,T)\to \mathbb{S}^2$, the general micromagnetic energy incorporating Dzyaloshinskii-Moriya  interaction (DMI) is defined by (also see \cite{MR4077401})
\begin{equation}\label{functional1}
	\mathcal{E}[\bm{m}] := \mathcal{E}_{\text{ex}}[\bm{m}] + \mathcal{E}_{\text{DMI}}[\bm{m}] + \mathcal{E}_{\text{lo}}[\bm{m}] + \mathcal{E}_{\text{appl}}[\bm{m},\bm{f}],
\end{equation}
where the individual energy contributions on the right-hand side are specified as follows.

\medskip
\noindent
\emph{Exchange energy.}
The exchange energy is given by
$$\mathcal{E}_{\text{ex}}[\bm{m}]:=\frac{\ell_{\text{ex}}^{2}}{2} \int_{\Omega} |\nabla \bm{m}|^{2} \mathrm{d}x,$$
where $\ell_{\text{ex}} > 0$ denotes the length of exchange.

\medskip
\noindent
\emph{Dzyaloshinskii-Moriya interaction.}
The DMI energy, with interaction strength $\kappa \in \mathbb{R}$, is defined by
\[
\mathcal{E}_{\text{DMI}}[\bm{m}]
:= \kappa \int_{\Omega} \bm{m} \cdot( \nabla \times \bm{m}) \,\mathrm{d}x.
\]

\medskip
\noindent
\emph{Lower-order energy.}
The lower-order energy functional takes the form
\[
\mathcal{E}_{\text{lo}}[\bm{m}]
:= -\frac{1}{2} \int_{\Omega} \bm{m} \cdot \bm{\pi}(\bm{m}) \,\mathrm{d}x,
\]
where the linear operator
\[\bm{\pi} : L^{2}(\Omega;\mathbb{R}^{3}) \to L^{2}(\Omega;\mathbb{R}^{3})\]
is required to be bounded and self-adjoint. This term models various lower-order effects, including magnetic anisotropy and stray-field interactions. More precisely, the magnetic anisotropy energy density $\Phi(\bm{m})$ is typically a homogeneous quadratic function of magnetization $\bm{m}$. For uniaxial materials with an easy axis aligned with the $x$-axis, one has
$$\Phi(\bm{m})=\bm{m}_2^2+\bm{m}_3^2.$$
The stray-field induced by $\bm{m}$ is defined as 
\[h_d(\bm{m})(x):=\int_{\Omega}\nabla N (x-y)\cdot \bm{m}(y)dy,\]
where $N(x) = -\frac{1}{4\pi|x|}$ is the Newtonian potential. Accordingly, the anisotropy energy and the stray-field energy can be written in the form
\[\frac{a}{2}\int_{\Omega} \Phi(\bm{m}) \,\mathrm{d}x+\frac{b}{2}\int_{\Omega} \bm{m} \cdot h_d(\bm{m}) \,\mathrm{d}x,\]
where $a,b$ are fixed material-dependent constants.

\medskip
\noindent
\emph{Applied-field energy.}
The applied-field energy is defined by
\[
\mathcal{E}_{\text{appl}}[\bm{m}]
:= -\int_{\Omega} \bm{m} \cdot \bm{f} \,\mathrm{d}x,
\]
which represents the interaction energy with a prescribed external field
$\bm{f} \in L^{2}(\Omega\times [0,T);\mathbb{R}^{3})$. A notable special case of $\mathcal{E}_{\text{lo}}[\bm{m}]+\mathcal{E}_{\text{appl}}[\bm{m}]$  is the external Zeeman energy:
\begin{align*}
	\mathcal{E}_{Z}[\bm{m}]:=&\frac{h}{2}\int_{\Omega}|\bm{m} -\bm{e}_3|^2  \mathrm{d}x\\
	=&h\int_{\Omega}(1-\bm{m}_3 )\mathrm{d}x 
        =h\int_{\Omega} \bm{m} \cdot (\bm{m}-\bm{e}_3) \mathrm{d}x,
\end{align*}
where $\bm{e}_3=(0,0,1)$ and $h$ is a positive constant.

\medskip
Associated with micromagnetic energy \eqref{functional1}, the Landau-Lifshitz-Gilbert (LLG) equation with a transport term reads
\begin{equation}\label{LLG-MDI}
	\partial_{t} \bm{m} + \gamma \bm{v} \cdot \nabla \bm{m}=
	-\alpha \bm{m} \times \bm{H} (\bm{m})-\beta \bm{m} \times \left(\bm{m} \times \bm{H}(\bm{m})\right),
\end{equation}
where $\bm{v}$ is a given time-dependent vector field, the effective magnetic field is given by
\begin{equation*}
	\bm{H}(\bm{m}):= -\frac{\partial\mathcal{E}[\bm{m}]}{\partial \bm{m}}= \ell_{\text{ex}}^{2} \Delta \bm{m} - 2\kappa \nabla \times \bm{m} + \bm{\pi}(\bm{m}) + \bm{f}.
\end{equation*}
Here, $\alpha \geq 0$ is the gyroscopic ratio and $\beta \geq 0$ denotes the Gilbert damping coefficient.

A prototypical special case of \eqref{LLG-MDI} is the classical Landau-Lifshitz (LL) equation 
\begin{align*}
	\partial_{t} \bm{m} = - \bm{m} \times \Delta \bm{m}, \quad  \text{in } \Omega \times [0,T), 
\end{align*}
which is a fundamental phenomenological model in the theory of ferromagnetism, describing the precessional dynamics of magnetization fields in ferromagnetic materials. It was first introduced by Landau and Lifshitz \cite{landau1935} in 1935, and later extended by Gilbert \cite{Gilbert1955,1353448} through the inclusion of a dissipative term, leading to the Landau-Lifshitz-Gilbert (LLG) equation
\begin{align*}
	\partial_{t} \bm{m} = - \alpha \bm{m} \times \Delta \bm{m}- \beta \bm{m} \times (\bm{m} \times \Delta \bm{m} ),\quad  \text{in } \Omega \times [0,T).
\end{align*}

Another important special case of \eqref{LLG-MDI} is the incompressible Schr\"odinger flow arising in hydrodynamic applications:
\begin{equation*}
	\partial_t\bm{m} + \gamma \bm{v}\cdot\nabla \bm{m} = -\bm{m} \times \Delta \bm{m},
\end{equation*}
where $\bm{v}$ is a divergence-free vector field. This system was derived by Chern et al.~~\cite{Chern2016} as a model for purely Eulerian simulation of incompressible fluids and has further applications in modeling magnetic fluids where quantum effects become significant.

\subsection{Related works} 
Due to their widespread applications in physics, the mathematical analysis of the LL-type equations has attracted considerable attention over the past several decades.

Zhou and Guo \cite{Zhou1984} established the existence of weak solutions to the boundary value problem of the LLG equation by means of the viscosity elimination method combined with a Leray-Schauder fixed point argument. Later, Visintin  \cite{Visintin1985} extended these results to models including magnetostrictive effects. In the dissipation-free case (i.e., the Landau-Lifshitz equation), Sulem, Sulem, and Bardos \cite{Sulem1986} employed difference methods to prove the global existence of weak solutions on $\mathbb{R}^n$. Subsequently, Wang \cite{W98} extended the global existence result of weak solutions to the Schr\"odinger flow from a closed Riemannian manifold or a bounded domain in $\mathbb{R}^n$ into $\mathbb{S}^2$, by using the complex structure approximation method. Further developments concerning weak solutions to a class of generalized flows related to the micromagnetic energy can be found in \cite{CW21,Chen2024Incompressible} and references therein. On the other hand, Alouges and Soyeur \cite{Alouges1992} demonstrated a non-uniqueness result for weak solutions to the LLG equation with Neumann boundary conditions on the unit ball in $\mathbb{R}^3$.

Next, we briefly recall the well-posedness of regular solutions for the LLG equation on domains with boundary. The local existence of regular solutions to the LLG equation on bounded domains in $\mathbb{R}^n$ ($n \leq 3$) was investigated by Carbon and Fabrie \cite{Carbon2001}. These results were later extended by Carbou and Jizzini \cite{Carbon2018} to models involving electric currents, with a detailed analysis of the associated boundary compatibility conditions. More recently, the local existence of regular solutions for the LLG equation with spin-polarized transport, as well as for the Schr\"{o}dinger flow with the damping term for maps from a 3-dimensional manifold with boundary into compact symplectic manifolds under the Neumann boundary conditions, was obtained in \cite{MR4619012, MR4686699}. The challenging problem of local existence and uniqueness of regular solutions for the initial-Neumann boundary value problem to the LL equation (i.e., the LLG equation with $\beta=0$) was recently resolved by the first author of the present paper and Wang \cite{Chen2023Regular, CW4, CW5}. By analyzing the interplay between the geometric structure of the LL equation and the transport effects induced by the vector field $\bm{v}$, the authors \cite{Chen2024Incompressible}  further established the local existence of regular solutions for the initial-Neumann boundary value problem of incompressible Schr\"{o}dinger flow on bounded domains in $\mathbb{R}^3$.

Beyond the study of existence theory, some progress has been made in understanding singularity formation of the LLG equation.  Ding and Wang \cite{DW07} obtained
the existence of a smooth finite time blow-up solution for LLG equations defined on a closed Riemmanian manifold  with or without boundary in dimensions  $n=3,4$. Harpes \cite{Har04} studied the behavior of finite time blow-up solutions of the 2-dimensional  LLG flow defined on surfaces. Recently,  Li and Song \cite{LS25} investigated finite-time singularities of 2-dimensional LLG equation for almost-holomorphic maps, and established a refined bubble tree convergence result. Precise  examples of finite-time blow-up solutions to 1-equivariant LLG equation from $\mathbb{R}^2$ into $\mathbb{S}^2$ were constructed by Xu and Zhao \cite{XZ25}. For the non-equivariant case, Wei, Zhang and Zhou \cite{WZZ22} constructed finite time blow-up solutions to the LLG equation from $\mathbb{R}^2$ into $\mathbb{S}^2$. They proved that there exist regular initial data such that the solution blows up precisely at these points at finite time, taking around each point the profile of a sharply scaled degree 1 harmonic map with the type II blow-up speed.  For the more  challenging case of the LLG equation without dissipation, Merle, Raphael and Rodnianski \cite{Merle2013} constructed finite-time blow-up solutions near harmonic maps for the 1-equivariant Schr\"odinger flow. Moreover, traveling wave solutions with vortex structures for Schr\"{o}dinger map flows were obtained by Lin and Wei \cite{Lin2010}, and later extended by Wei and Yang \cite{Wei2016}.

Considerable attention has also been devoted to the LLG equation with DMI. D\"oring and Melcher in \cite{MR3639614} and Shimizu in \cite{Shimizu} discussed the local well-posedness for this equation on $\mathbb{R}^2$.  The long-time existence and stability of dynamical solutions to the one-dimensional  LLG equation with DMI, subject to a class of sufficiently weak external driving fields, were established in \cite{MR4616656}. The global existence of smooth solutions with small initial data on the two-dimensional flat torus and on $\mathbb{R}^2$ was proved in \cite{MR4730566}. Recently, \citet{MR4746009} obtained the global existence of smooth solutions in one dimension for large initial data, and also established the global existence of a weak solution with homogeneous Neumann boundary on bounded domains in $\mathbb{R}^2$ or $\mathbb{R}^3$ using a penalized approximation approach. 

From a physical perspective, when the magnetization field $\bm{m}$ is subject to DMI, one is naturally led to the chiral boundary value problem for the LLG equation:
\begin{equation}\label{eq:llg}
	\begin{cases}
		\partial_{t} \bm{m} + \gamma \bm{v} \cdot \nabla \bm{m}
		=-\alpha \bm{m} \times \bm{H} (\bm{m})-\beta \bm{m} \times \left(\bm{m} \times \bm{H}(\bm{m})\right),\quad\quad&\text{in } \Omega \times [0,T),\\[1ex]
		\bm{m}(\cdot,0) = \bm{m}_{0},  \quad &\text{in } \Omega,
	\end{cases}
\end{equation}
together with the chiral boundary condition (also called the helical boundary condition):
\begin{align}\label{chiralBC}
    \ell_{\text{ex}} \partial_{\bm{n}}  \bm{m} + \frac{\kappa}{\ell_{\text{ex}} }  \bm{m} \times \bm{n}= 0, \quad \text{on } \partial \Omega \times [0,T).
\end{align}
Here, $\Omega$ is a smooth bounded domain in $\mathbb{R}^m$ with $m=2,3$, and $\bm{n}$ denotes the unit outward normal vector on $\partial \Omega$. Using ideas from topological fluid dynamics and the theory of liquid crystals, Fratta, Innerberger and Praetorius~\cite{MR4077401} proved a weak-strong uniqueness result for the problem \eqref{eq:llg}-\eqref{chiralBC} in the case of vanishing $\bm{v}$. Subsequently, Fratta et~al. \cite{DFP} analyzed the thin-film limit of the micromagnetic energy functional with DMI and characterized the asymptotic behavior of the LLG equation~\eqref{eq:llg}.

To the best of our knowledge, the literature contains only limited results concerning the existence of weak or regular solutions to the chiral boundary value problem \eqref{eq:llg}-\eqref{chiralBC}. This provides the main motivation for the present work, in which we establish the global existence of weak solutions to this problem. We also plan to investigate the well-posedness of regular solutions in future work.

\subsection{The main results}
In this part, we present our main results on the global existence of weak solutions to the problem \eqref{eq:llg}-\eqref{chiralBC}.

\subsubsection{\bf{LLG equation with DMI in the framework of helical derivatives}} For the purpose of PDE analysis, it is convenient to rewrite the LLG equation with DMI in the framework of helical derivatives. For simplicity, we set $\ell_{\text{ex}}=1$ and $\kappa=1$. Let $\bm{u} \in H^1(\Omega)$. The partial helical derivative of $\bm{u}$ is defined by
\begin{align}\label{partialD}
    \partial_i^{\mathfrak{h}}\bm{u} :=\partial_i\bm{u}-\bm{e}_i \times \bm{u},
\end{align}
where $\bm{e}_i \in \mathbb{R}^3$ denotes the $i$-th canonical basis vector for $i = 1,2,3$. Consequently, the gradient and the Laplacian of $\bm{u}$ with respect to helical derivatives are defined by  
\begin{align}\label{partailL}
	\nabla^{\mathfrak{h}}\bm{u} := (\partial_1^{\mathfrak{h}}\bm{u}, \partial_2^{\mathfrak{h}}\bm{u}, \partial_3^{\mathfrak{h}}\bm{u}), 
	\quad 
	\Delta^{\mathfrak{h}}\bm{u} := \sum_{i=1}^3 \partial_i^{\mathfrak{h}}\partial_i^{\mathfrak{h}}\bm{u},
\end{align}
respectively. 

Therefore,  the micromagnetic energy \eqref{functional1} admits the following equivalent representation (see the formula \eqref{eq:helical_energy_decomposition}):
\begin{equation}\label{helical_total_energy}
	\mathcal{E}^{\mathfrak{h}}[\bm{m}] := \frac{1}{2}\int_{\Omega} |\nabla^{\mathfrak{h}} \bm{m}|^{2} \mathrm{d}x - \frac{1}{2} \int_{\Omega} \bm{m} \cdot \bm{\pi}(\bm{m})\, \mathrm{d}x - \int_{\Omega} \bm{m} \cdot \bm{f}\, \mathrm{d}x.
\end{equation}
Since the domain $\Omega$ is bounded, this energy functional differs from \eqref{functional1} only by a finite constant.

By using the  identities in Proposition \ref{lemma2.1}, the chiral boundary value problem \eqref{eq:llg}-\eqref{chiralBC} is equivalent to the following natural boundary problem of LLG equation with helical derivatives:
\begin{equation}\label{system1}
	\begin{cases}
		D^{\mathfrak{h}}_t\bm{m}=-\alpha \bm{m} \times \bm{H}(\bm{m})-\beta \bm{m} \times \left(\bm{m} \times \bm{H}(\bm{m})\right), \quad & \text{in } \Omega \times [0,T), \\[1ex]
		\nabla^{\mathfrak{h}}  \bm{m} \cdot\bm{n} = 0, \quad & \text{on } \partial \Omega \times [0,T),\\[1ex]
		\bm{m}(\cdot,0) = \bm{m}_{0},  \quad &\text{in } \Omega,
	\end{cases}
\end{equation}
where the effective field in the helical formulation is given by 
\begin{align}\label{heff}
	\bm{H}(\bm{m})= 
	\Delta^{\mathfrak{h}}\bm{m} + \bm{\pi}(\bm{m}) + \bm{f},
\end{align}
and the transport time derivative of $\bm{m}$ is defined by 
\[D^{\mathfrak{h}}_t\bm{m}:=\partial_{t} \bm{m} + \gamma (\bm{v} \cdot \nabla^{\mathfrak{h}}\bm{m}+\bm{v} \times \bm{m}).\]

Since $\bm{m}\times$ can be interpreted as a complex structure $J(\bm{m})=\bm{m}\times: T_{\bm{m}}\mathbb{S}^2\to T_{\bm{m}}\mathbb{S}^2$,  the problem \eqref{system1} can be rewritten as the following version:
\begin{equation}\label{system2}
	\begin{cases}
		\alpha D^{\mathfrak{h}}_t\bm{m}-\beta\bm{m}\times D^{\mathfrak{h}}_t\bm{m}=-(\alpha^2+\beta^2) \bm{m} \times \bm{H}(\bm{m}), \quad & \text{in } \Omega \times [0,T), \\[1ex]
		\nabla^{\mathfrak{h}}  \bm{m} \cdot\bm{n} = 0, \quad & \text{on } \partial \Omega \times [0,T),\\[1ex]
		\bm{m}(\cdot,0) = \bm{m}_{0},  \quad &\text{in } \Omega.
	\end{cases}
\end{equation}

\subsubsection{\bf{Main results on LLG equation with helical derivatives}} In this part, we start with the definition of weak solutions to the problem \eqref{system2}.

\begin{definition}\label{definition0}
Let $\Omega$ be a bounded domain in $\mathbb{R}^3$ and  assume $ \bm{m}_0 \in H^1(\Omega; \mathbb{S}^2)$. A map $\bm{m}$ is called a global type-I weak solution to the problem \eqref{system1}, if for any $T<\infty$, $\bm{m} \in H^1([0,T]\times\Omega; \mathbb{S}^2)$ satisfies 
		\begin{align*}
			&\alpha\int_0^T \int_{\Omega}  D^{\mathfrak{h}}_t\bm{m}\cdot\bm{\phi}{\rm d } x {\rm d }t-\beta\int_0^T \int_{\Omega}  \bm{m}\times D^{\mathfrak{h}}_t\bm{m}\cdot\bm{\phi}{\rm d } x {\rm d }t\\
			=&(\alpha^2+\beta^2)\int_0^T \int_{\Omega}(\bm{m}\times \nabla^\mathfrak{h} \bm{m})\cdot\nabla^\mathfrak{h} \bm{\phi} {\rm d}x{\rm d}t
			-(\alpha^2+\beta^2) \int_0^T \int_{\Omega} (\bm{m}\times (\bm{\pi}(\bm{m})+ \bm{f}))\cdot  \bm{\phi} {\rm d}x{\rm d}t,
		\end{align*}
		for all $\bm{\phi} \in H^1([0,T]\times \Omega)$. Here, $\bm{m}(x,t)\rightarrow \bm{m}_0$ as $t \rightarrow 0$ in $C^0([0,T]; L^2(\Omega))$.
		
\end{definition}
Our first result concerning  the LLG equation with the helical derivatives can be stated as follows.
\begin{theorem}\label{theroem1}
	Let $\Omega\subset \mathbb{R}^3$ be a bounded smooth domain. Assume that
	$\bm{m}_0 \in H^1(\Omega; \mathbb{S}^2), 
	\bm{f}\in L^2_{loc}([0,\infty);L^2(\Omega))$ is locally Lipschitz, and $\bm{v}\in L^2_{loc}([0,\infty); L^\infty(\Omega))$.
	In addition, suppose that there exists a universal constant $C_{\bm{\pi}}$ such that 
	\begin{align}\label{Cpi1}
		\|\bm{\pi}(\bm{u})\|_{L^{2}(\Omega)} \leq C_{\bm{\pi}} \|\bm{u}\|_{L^{2}(\Omega)}, \quad \forall  \bm{u} \in L^{2}({\Omega} ;\mathbb{R}^{3}).
	\end{align}
	Then problem \eqref{system2} admits a global type-I weak solution $\bm{m}$ in the sense of Definition \ref{definition0}, which satisfies
	\begin{align*}
		\bm{m}\in L^\infty([0,T];H^1(\Omega;\mathbb{S}^2)),\quad \partial_{t}\bm{m}\in L^2([0,T]\times \Omega) 
	\end{align*}
	for any $T<\infty$.
\end{theorem}

\begin{remark}
\begin{itemize}
\item[$(1)$] The type-I weak solution $\bm{m}$ given in Theorem \ref{theroem1} satisfies the energy inequality \eqref{energy-ineq}. In the special case where $\ga=0$, we have
		\begin{align*}
			\mathcal{E}^{\mathfrak{h}}[\bm{m}(T)]+\int_0^T\int_{\Omega}\bm{m}\cdot \partial_t \bm{f}{\rm d}x
			{\rm d}t+\frac{\beta }{\alpha^2+ \beta^2} \int_0^T\int_{\Omega}\left|\partial_t\bm{m}\right|^2{\rm d}x\leq \mathcal{E}^{\mathfrak{h}}[\bm{m}_0].
		\end{align*}
	\item[$(2)$] Fratta et~al. \cite*{MR4077401} considered the LLG equation \eqref{system2} with a vanishing vector field $\bm{v}$, namely,
	\begin{align*}
		\partial_{t} \bm{m}-\tilde{\alpha}\bm{m} \times \partial_{t} \bm{m}=-\bm{m} \times \bm{H}(\bm{m}), \quad {\rm in } \:\: \Omega \times [0,T),
	\end{align*}
	where we choose $\al=\frac{1}{\tilde{\alpha}^2+1}$ and $\beta=\frac{\tilde{\alpha}}{\tilde{\alpha}^2+1}$ in \eqref{system2}. They established a weak-strong uniqueness property: any type-I weak solution in the sense of Definition~\ref{definition0} coincides with a $C^3$-smooth solution, provided that the operator $\bm{\pi}$ satisfies the following additional regularity assumption: 
	\begin{equation}\label{regularityoff}
		\|\bm{\pi}(\bm{u})\|_{L^{\infty}([0,T]; C^{1}(\overline{\Omega}))} < \infty,\quad \forall  \bm{u} \in C^{3}(\overline{\Omega} \times [0,T]; \mathbb{R}^{3}).
	\end{equation}
	Consequently,  Theorem~\ref{theroem1} confirms the global existence of such weak solutions.
\end{itemize}
\end{remark}

We can also obtain the global weak solutions with lower regularity than those obtained in Theorem~\ref{theroem1} for problem~\eqref{system1}, provided that the vector field $\bm{v}$ satisfies an alternative regularity assumption \eqref{(a)}:
\begin{equation}\label{(a)}
	\bm{v} \in  L^2_{loc} ([0,\infty); L^3(\Omega)\cap W^{1,1}(\Omega)),\quad  \nabla \bm{v } =-( \nabla \bm{v })^T\,\,\text{in} \,\, \Omega\times [0, \infty),\tag{a}
\end{equation}
together with the boundary condition $\bm{v}\cdot \bm{n}|_{\partial \Omega \times [0,\infty)}=0$. Here, we view $\nabla \bm{v}$ as a $3\times 3$-matrix, and $( \nabla \bm{v })^T$ denotes its transpose.

Under the assumption \eqref{(a)}, we introduce the notion of a weaker solution to problem~\eqref{system2} as follows.
\begin{definition}\label{definition}
	Let $\Omega$ be a bounded domain in $\mathbb{R}^3$. Assume $ \bm{m}_0 \in H^1(\Omega;\mathbb{S}^2)$, a map $\bm{m}$ is called a global type-II weak solution to the problem \eqref{system1} if, for any $T<\infty$, we have
	\begin{align*}
		\bm{m} \in L^2 ([0,T];H^1(\Omega;\mathbb{S}^2))\cap C^0([0,T];L^2(\Omega;\mathbb{S}^2)), 
	\end{align*}
	and $\bm{m}$ satisfies 
	\begin{align*}
		&\int_{\Omega} (\bm{m} \cdot \bm{\phi })(T){\rm d \: }x - \int_{\Omega} (\bm{m} \cdot \bm{\phi} )(0)\: {\rm d }x -\int_0^T \int_{\Omega}  \bm{m} \cdot \frac{\partial \bm{\phi}}{\partial t }\:{\rm d } x {\rm d }t\\
		=&-\gamma\int_0^T \int_{\Omega}(\bm{v} \cdot \nabla^{\mathfrak{h}} \bm{m}+\bm{v} \times \bm{m})\cdot  \bm{\phi} \:{\rm d}x{\rm d}t+\alpha \int_0^T \int_{\Omega}(\bm{m}\times \nabla^\mathfrak{h} \bm{m})\cdot\nabla^\mathfrak{h} \bm{\phi} \:{\rm d}x{\rm d}t\\
		&-\beta \int_0^T \int_{\Omega} 
		\nabla^\mathfrak{h}  \bm{m}\cdot \nabla^\mathfrak{h}\bm{\phi} 
		\:{\rm d}x{\rm d}t+\beta \int_0^T \int_{\Omega} 
		|\nabla^\mathfrak{h}\bm{m}|^2 \bm{m}\cdot \bm{\phi} 
		\:{\rm d}x{\rm d}t\\
		&-\alpha \int_0^T \int_{\Omega} (\bm{m}\times (\bm{\pi}(\bm{m})+ \bm{f}))\cdot  \bm{\phi} \:{\rm d}x{\rm d}t
        -\beta \int_0^T \int_{\Omega} (\bm{m} \times  (\bm{m} \times (\bm{\pi }(\bm{m})+\bm{f})))\cdot\bm{\phi} \:{\rm d}x{\rm d}t
	\end{align*}
	for all $\bm{\phi} \in H^1(\Omega \times [0,T])\cap L^\infty(\Omega \times [0,T])$. Here, $\bm{m}(x,t)\rightarrow \bm{m}_0$ as $t \rightarrow 0$ in $C^0([0,T]; L^2(\Omega))$.
\end{definition}

Then we have the following result.
\begin{theorem}\label{theroem1'}
	Let $\Omega\subset \mathbb{R}^3$ be a bounded smooth domain. Assume that
	$\bm{m}_0 \in H^1(\Omega; \mathbb{S}^2), 
	\bm{f}\in L^2_{loc}([0,\infty);L^2(\Omega))$ is locally Lipschitz and that $\bm{v}$ satisfies the condition \eqref{(a)}.
	In addition, suppose that there exists a universal constant $C_{\bm{\pi}}$ such that 
	\begin{align}\label{Cpi}
		\|\bm{\pi}(\bm{u})\|_{L^{2}(\Omega)} \leq C_{\bm{\pi}} \|\bm{u}\|_{L^{2}(\Omega)}, \quad \forall  \bm{u} \in L^{2}({\Omega} ).
	\end{align}
	Then problem \eqref{system1} admits a global type-II weak solution $\bm{m}$ in the sense of Definition \ref{definition}, which satisfies
	\begin{align*}
		\bm{m}\in L^\infty([0,T];H^1(\Omega;\mathbb{S}^2)),\quad \partial_{t}\bm{m}\in L^2([0,T]; H^{-1}(\Omega)) 
	\end{align*}
	for any $T<\infty$.
\end{theorem}

\begin{remark}
There exists exactly a vector field satisfying 
	$$\nabla \bm{v } =-( \nabla \bm{v })^T.$$
A classical example is the solid body rotation flow $\bm{v}$ (see \cite{1990Elementary}), which satisfies the steady Euler equation in the form
	\begin{align*}
		\bm{v} =\bm{\omega} \times \bm{x}, \quad p = \frac{\rho}{2} |\bm{\omega}\times\bm{x}|^2 + \text{const},
	\end{align*}
where $ \bm{\omega}$ denotes a constant angular velocity, $\rho$ is the density of flow, $p$ is the pressure function, and $ \bm{x}=(x_1,x_2,x_3)$. It is worth pointing out that the general time-dependent rigid body flow
	\[\bm{v}(\bm{x},t)=\bm{a}(t)+\bm{\omega}(t)\times \bm{x}\]
also meets the condition $\nabla \bm{v } =-( \nabla \bm{v })^T$ at each fixed time $t$. 
	
\end{remark}

\subsubsection{\bf{The incompressible Schr\"odinger flow with helical derivatives}} 

Our final theorem concerns the global existence of solutions to the chiral boundary value problem for the incompressible Schr\"odinger flow (i.e., the LLG equation \eqref{LLG-MDI} without dissipative term) with helical derivatives:
\begin{equation}\label{eq:IS}
	\begin{cases}
		D^{\mathfrak{h}}_{t} \bm{m}
		= -\alpha \bm{m} \times \bm{H}(\bm{m}), \quad & \text{in } \Omega \times [0,T), \\[1ex]
		\nabla^{\mathfrak{h}}  \bm{m} \cdot\bm{n} = 0, \quad & \text{on } \partial \Omega \times [0,T),\\[1ex]
		\bm{m}(\cdot,0) = \bm{m}_{0},  \quad &\text{in } \Omega.
	\end{cases}
\end{equation}

Assume that the vector field $\bm{v}$ satisfies either condition  \eqref{(b)}:
\begin{align}\label{(b)}
	\bm{v} \in  L^2_{loc}([0,\infty); H^1(\Omega)), \:  \:\bm{v}\cdot \bm{n}|_{\partial \Omega \times  [0, \infty)}=0, \: \:\nabla \bm{v } =-( \nabla \bm{v })^T\,\,\text{in} \,\, \Omega\times [0, \infty),
	\tag{b}
\end{align} 
or condition \eqref{(c)}:
\begin{align}\label{(c)}
	\bm{v} \in  L^2_{loc}([0,\infty); L^3(\Omega)), \:  \nabla \bm{v} \in  L^1_{loc} ([0,\infty); L^\infty(\Omega)), \: \bm{v}\cdot \bm{n}|_{\partial \Omega \times  [0, \infty)}=0, \: {\rm div} \bm{v}=0 \,\,\text{in} \,\, \Omega\times [0, \infty).
	\tag{c}
\end{align} 

Under these assumptions, we obtain the following global existence result.
\begin{theorem}\label{theorem2}
	Let $\Omega\subset \mathbb{R}^3$ be a bounded smooth domain. Assume that the initial data 
	$\bm{m}_0 \in H^1(\Omega; \mathbb{S}^2), 
	\bm{f}\in L^2_{loc}([0,\infty); H^1(\Omega))$ is locally Lipschitz, and that the vector field $\bm{v}$ satisfies condition \eqref{(b)} or \eqref{(c)}. Moreover, suppose that there exists a universal constant $C_{\bm{\pi}} > 0$ such that
	\begin{equation}\label{C-pi1}
		\|\bm{\bm{\pi}}(\bm{u})\|_{H^{1}(\Omega)} \leq C_\pi \|\bm{u}\|_{H^{1}(\Omega)}, \quad \forall  \bm{u} \in L^{2}({\Omega}).
	\end{equation}
	Then problem \eqref{eq:IS} admits a global type-II weak solution $\bm{m}$ in the sense of Definition \ref{definition} with $\beta=0$, satisfying
	\begin{align*}
		\bm{m}\in L^\infty([0,T];H^1(\Omega;\mathbb{S}^2)),\quad \partial_{t}\bm{m}\in L^2([0,T]; H^{-1}(\Omega)) 
	\end{align*}
	for any $T<\infty$.
\end{theorem}

\begin{remark}
	It is worth pointing out that the above global existence results (i.e., Theorem \ref{theroem1}, Theorem \ref{theroem1'} and Theorem \ref{theorem2}) are also valid for the domain
	$\Omega \subset \mathbb{R}^{2}$. 
	
	Indeed, in this case one may regard $\Omega$ as a planar domain embedded in $\mathbb{R}^{3}$ and define the gradient operator by 
	\[\nabla = \left(\dfrac{\partial}{\partial x_1}, \dfrac{\partial}{\partial x_2},0\right).\]
	With this convention, the same arguments as those used in the proofs of Theorems~\ref{theroem1}, \ref{theroem1'}, and~\ref{theorem2} apply verbatim, thereby yielding the corresponding two-dimensional results.
\end{remark}

\subsubsection{\bf{Outline of proof}} 

These three theorems are proved by employing the complex structure approximation method originally introduced in \cite{W98}. In contrast to the penalized approach used in \cite{MR4746009}, the innovation of the present method lies in its ability to preserve the geometric structures of the equations throughout the approximation process, making it particularly well-suited for studying the chiral boundary value problem of the LLG equation (i.e., \eqref{system1}) with or without dissipation.

More precisely, for any $\bm{m}\in \mathbb{S}^2$, the map $\bm{m}\times: T_{\bm{m}}\mathbb{S}^2 \to T_{\bm{m}}\mathbb{S}^2$ defines a complex structure on $\mathbb{S}^2$, which induces the following divergence structure in equation \eqref{system1}:
\[\bm{m}\times \De^{\mathfrak{h}} \bm{m}=\mbox{div}^{\mathfrak{h}}(\bm{m}\times \n^{\mathfrak{h}}\bm{m}).\]
This property constitutes the first crucial ingredient to obtain globally weak solutions to \eqref{system1} and to derive their a priori estimates. To preserve this structure, we introduce the following approximation of the complex structure $\bm{m}\times$:
\[\dfrac{\bm{m}}{\max\{|\bm{m}|,1\}}\times.\]
Based on this, we then construct suitable approximation equations for \eqref{system1} (see \eqref{I-galerkin1} and \eqref{II-galerkin1}) that retain the divergence structure.

The second key ingredient, which is also a novel aspect of our approach, is to establish the compactness of the approximate solutions. To this end, we introduce Sobolev spaces adapted to the helical derivative and establish uniform energy estimates compatible with the chiral boundary condition. These estimates, which may be of independent interest, yield the compactness.

\medskip
The remainder of this paper will be organized as follows. In  \autoref{sec:preliminaries}, we
introduce the notation for Sobolev spaces adapted to the framework of helical derivatives and collect several preliminary lemmas. In \autoref{sec:main}, we prove Theorem \ref{theroem1}. Theorems \ref{theroem1'} and \ref{theorem2}
are established in \autoref{s: Type-II weak solution}.  Finally, in \autoref{appendix}, we derive several fundamental properties and regularity estimates for maps in Sobolev spaces involving the helical derivatives.

\section{Preliminaries}\label{sec:preliminaries}

\subsection{Sobolev spaces with the helical derivative}
We start by introducing some notation and the basic properties of Sobolev space associated with the helical derivative.  Let $\Omega $ be a bounded smooth domain in $\mathbb{R}^3$, and $ \bm{u}=(u_1, u_2, u_3): \Omega \to \mathbb{R}^3$ be a vector-valued function.  For each $k\in \mathbb{N}$, let  $H^k(\Omega):=W^{k,2}(\Om; \mathbb{R}^3)$ be the usual Sobolev space. Then we define  
\begin{align*}
	H^k(\Omega; \mathbb{S}^2) := 
	\left\{ \bm{u}\in H^k(\Omega)\:\big|
	\:|\bm{u}|=1, \: {\rm a.e.} \: x \in \Omega \right\}.
\end{align*}
For simplicity, $H^{-1}(\Omega)$ denotes the dual space of the Sobolev space $H^1_0(\Omega)$.

Recalling the definitions of the partial helical derivatives \eqref{partialD}, the helical gradient, and the helical Laplacian \eqref{partailL}, we define the Sobolev space adapted to the helical derivative. 

\begin{definition}
	Let $\Omega\subset \mathbb{R}^3$ be a bounded smooth  domain. If $ \bm{u} \in H^k(\Omega) $ for $k\in \mathbb{N} $, we define the helical Sobolev norm of $\bm{u}$ by
	\[
	\|\bm{u}\|_{H_{\mathfrak{h}}^k(\Omega)} := \left(\sum_{i=0}^k\int_\Omega \left|(\n^{\mathfrak{h}})^i \bm{u}\right|^2  \mathrm{d} x\right)^{\frac{1}{2}}.
	\]
\end{definition}

The following two propositions on Sobolev space with the helical derivative are proved in Appendix \ref{appendix}.
\begin{proposition}\label{proposition2.1}
	There exists a positive constant $C=C(\Omega)$ such that for any $\bm{u}\in H^k(\Om)$ with $k=1,2$, we have
	\begin{align}
		C^{-1} \norm{\bm{u}}_{H^k_{\mathfrak{h}}(\Omega)} \leq \norm{\bm{u}}_{H^k(\Omega)} \leq C \norm{\bm{u}}_{H^k_{\mathfrak{h}}(\Omega)}.\label{equiv-1}
	\end{align}
\end{proposition}

Using Proposition \ref{proposition2.1}, we establish the following equivalent estimates of Sobolev norms. 

\begin{proposition}\label{proposition2}
	Let $\Omega$ be a bounded smooth domain in $\mathbb{R}^3$. There exists a positive constant $C=C(\Omega)$  such that for all $\bm{u} \in H^2(\Omega)$ with $\nabla^\mathfrak{h} \bm{u} \cdot \bm{n}=0$, we have
	\begin{align*}
		\norm{\bm{u}}_{H^2_{\mathfrak{h}}(\Omega)}
		\leq C 
		\left(
		\norm{\Delta^{\mathfrak{h}}\bm{u}}_{L^2(\Omega)} + 
		\norm{\bm{u}}_{L^2(\Omega)} 
		\right).
	\end{align*}
\end{proposition}

The following proposition collects some basic  properties of the helical derivative. We refer to 
\cite{MR3269033,MR4077401} for a proof.

\begin{proposition}\label{lemma2.1}
	The helical derivative satisfies the following properties.
	
	\begin{enumerate}
    \item For $\bm{u}_1, \bm{u}_2 \in H^1(\Omega;\mathbb{R}^3)$, we have the Leibniz rule
		\begin{equation}
			\nabla^{\mathfrak{h}}(\bm{u}_1 \times \bm{u}_2) =  \nabla^{\mathfrak{h}}\bm{u}_1 \times \bm{u}_2 + \bm{u}_1 \times  \nabla^{\mathfrak{h}}\bm{u}_2.
		\end{equation} 
		\item For $\bm{u}_1 \in H^1(\Omega;\mathbb{R}^3)$ and $\bm{u}_2 \in H^2(\Omega;\mathbb{R}^3)$, we have the integration by parts formula
		\begin{equation}
			\langle  \nabla^{\mathfrak{h}}\bm{u}_1,  \nabla^{\mathfrak{h}}\bm{u}_2 \rangle_{\Omega} = \langle \bm{u}_1,  \nabla^{\mathfrak{h}}\bm{u}_2 \cdot \bm{n} \rangle_{\partial\Omega} - \langle \bm{u}_1,  \Delta^{\mathfrak{h}}\bm{u}_2 \rangle_{\Omega},
		\end{equation}
		where $\bm{n}$ is the outward unit normal vector on $\partial\Omega$. Especially,
		\begin{equation}
			\langle  \nabla^{\mathfrak{h}}  \boldsymbol{u}_1 \times \bm{u}_2,\nabla^\mathfrak{h}\boldsymbol{u}_2\rangle_\Omega
			=
			\langle  \boldsymbol{u}_1 \times  \boldsymbol{u}_2, \nabla^{\mathfrak{h}} \boldsymbol{u}_{2} \cdot \boldsymbol{n} \rangle_{\partial\Omega}
			-
			\langle \boldsymbol{u}_2 \times \Delta^{\mathfrak{h}} \boldsymbol{u}_2, \boldsymbol{u}_1 \rangle_\Omega .
		\end{equation}

		\item For  $\bm{u} \in C^2(\Omega \times (0,T);\mathbb{R}^3)$, we have 
		\begin{gather}
			\frac{1}{2} |\nabla^{\mathfrak{h}}\bm{u}|^{2}
			= \frac{1}{2}|\nabla\bm{u}|^{2} +\bm{u}\cdot(\nabla \times\bm{u}) +|\boldsymbol{u}|^{2}, \label{eq:helical_energy_decomposition}
			\\
			\Delta^{\mathfrak{h}}\bm{u} =  \Delta \bm{u} -2\nabla \times \bm{u} -2 \bm{u}\label{eq:helical_energy_decomposition2},
			\\
			\bm{u}\times \Delta^{\mathfrak{h}} \bm{u}  =  \sum_{i=1}^{3} \partial_{i}^{\mathfrak{h}} (\bm{u} \times \partial_i^{\mathfrak{h}}\bm{u}),
			\\
            \partial_i^{\mathfrak{h}}\partial_k^{\mathfrak{h}}\boldsymbol{u} - \partial_k^{\mathfrak{h}}\partial_i^{\mathfrak{h}}\boldsymbol{u} =\left[ (\boldsymbol{e}_i \cdot \boldsymbol{u}) \boldsymbol{e}_k - (\boldsymbol{e}_k \cdot \boldsymbol{u}) \boldsymbol{e}_i \right].\label{eq:commutor_estimate}
		\end{gather}
	\end{enumerate}
\end{proposition}

\subsection{The basis of the Galerkin approximation} 
Let $\Omega$ be a bounded smooth domain in $\mathbb{R}^3$. Let $\lambda_i$ be the $i$-th eigenvalue of the operator $I - \Delta^{\mathfrak{h}}$ with respect to the helical boundary condition, and let $\bm{f}_i$ be the corresponding vector eigenfunction, namely
\begin{equation*}
	\begin{cases}
		(I - \Delta^{\mathfrak{h}} )\bm{f_i}=\lambda_i \bm{f}_i, \quad &x \in \Omega,\\
		\nabla^{\mathfrak{h}}\bm{f}_i \cdot \bm{n}=0, \quad &x \in \partial \Omega.
	\end{cases}
\end{equation*}

Since $I- \Delta^\mathfrak{h} $ is  a self-adjoint Fredholm operator (see Theorem 4.6 in \cite{MR2030823}), we may, without loss of generality, assume that
$\{\bm{f}_i\}_{i=1}^\infty$
forms a complete standard orthogonal basis of $L^2(\Omega;\mathbb{R}^3)$. 
Let
$\mathcal{H}_n:={\rm span}\{\bm{f}_1,\bm{f}_2,\cdots,\bm{f}_n\}$
be a finite subspace of 
$L^2(\Omega;\mathbb{R}^3)$, and $\mathcal{P}_n:L^2(\Omega;\mathbb{R}^3)\rightarrow \mathcal{H}_n$ 
be the Galerkin projection, which is a self-adjoint operator. In fact, for any $\bm{f}\in L^2(\Omega;\mathbb{R}^3)$, we have
$
\mathcal{P}_n \bm{f}=\sum_{i=1}^n\left\langle \bm{f},\bm{f}_i\right\rangle _{L^2(\Omega)}\bm{f}_i
$
and thus $\lim_{n\rightarrow\infty}\|\bm{f}- \mathcal{P}_n \bm{f}\|_{L^2(\Omega)}=0$.

\subsection{Auxiliary lemma on compactness}
Let $(\mathcal{X},\norm{\cdot }_{\mathcal{X}}) $ be a Banach space, and $f: [0,T] \to \mathcal{X}$ be a map. For any $p\geq1$ and $T>0$, we define
the functional space 
\begin{align*}
	L^p ([0,T];\mathcal{X})
	:= 
	\left\{ 
	f : [0,T]\to \mathcal{X}
	\:\big|\: \| f\|_{L^p ([0,T];\mathcal{X})}<\infty
	\right\}
\end{align*}
with the norm 
\begin{align*}
	\| f\|_{L^p ([0,T];\mathcal{X})}
	:= 
	\left(
	\int_0^{T} \| f\|_{\mathcal{X}}^p {\rm d}t
	\right)^\frac{1}{p}.
\end{align*}
For later application, we also need to recall the following compactness result. 
\begin{lemma}[see \cite{MR2986590} or  \cite{MR916688}]\label{LAS}
	Let $\mathcal{X}\subset \mathcal{B}\subset \mathcal{Y}$ be Banach spaces. Suppose that the embedding $\mathcal{B}\hookrightarrow \mathcal{Y}$ is continuous and that the embedding $\mathcal{X}\hookrightarrow \mathcal{B}$ is compact. Let $1\leq p,q,r\leq\infty$. For $T>0$, we define
	\[
	\mathcal{E}_{p,r} = \left\{f\in L^p([0,T];\mathcal{X}),\partial_t f\in L^r([0,T];\mathcal{Y})\right\}
	\]
	 equipping with a norm $\|f\|_{\mathcal{E}_{p,r}} := \|f\|_{L^p([0,T];\mathcal{X})} + \left\|\partial_t f\right\|_{L^r([0,T];\mathcal{Y})}$. Then, the following properties hold true.
	\begin{enumerate}
		\item[(1)] If $p<\infty$, then the embedding $\mathcal{E}_{p,r}$ in $L^p([0,T];\mathcal{B})$ is compact;
		\item[(2)] If $p<\infty$ and $p<q$, the embedding $\mathcal{E}_{p,r}\cap L^q([0,T];\mathcal{B})$ into $L^s([0,T];\mathcal{B})$ is compact for all $1\leq s<q$;
		\item[(3)] If $p=\infty$ and $r>1$, the embedding of $\mathcal{E}_{p,r}$ in $C^0([0,T];\mathcal{B})$ is compact.
	\end{enumerate}
\end{lemma}

\section{Type-I weak solutions to LLG equation}
\label{sec:main}
In this section, we establish the global existence of type-I weak solutions to the problem \eqref{system1}, namely, 
\begin{equation}\label{system1'}
	\begin{cases}
		D^{\mathfrak{h}}_t\bm{m}=-\alpha \bm{m} \times \bm{H}(\bm{m})-\beta \bm{m} \times \left(\bm{m} \times \bm{H}(\bm{m})\right), \quad & \text{in } \Omega \times [0,\infty), \\[1ex]
		\nabla^{\mathfrak{h}}  \bm{m} \cdot\bm{n} = 0, \quad & \text{on } \partial \Omega \times [0,\infty),\\[1ex]
		\bm{m}(\cdot,0) = \bm{m}_{0}:\Om\to \mathbb{S}^2,  &\text{in } \Omega,
	\end{cases}
\end{equation}
where we denote the transport time derivative of $\bm{m}$ by 
 \[D^{\mathfrak{h}}_t\bm{m}:=\partial_{t} \bm{m} + \gamma (\bm{v} \cdot \nabla^{\mathfrak{h}}\bm{m}+ \bm{v} \times \bm{m}).\]

%
Inspired by \cite{W98}, we consider the following approximate system of problem \eqref{system1'}: 
\begin{equation}\label{I-galerkin1}
	\begin{cases}
		\begin{aligned}
			\partial_{t} \bm{m} - \varepsilon\Delta^{\mathfrak{h}} \bm{m}
			&=\gamma \mathcal{J}(\bm{m})\times \left( \mathcal{J}(\bm{m})\times \left(\bm{v} \cdot \nabla^{\mathfrak{h}} \bm{m} + \bm{v} \times \mathcal{J}(\bm{m})\right)\right)
			\\
			&\quad 
			-\alpha \mathcal{J}(\bm{m}) \times  
			\bm{H}(\bm{m}) - \beta  \mathcal{J}(\bm{m}) \times
			\left( 
			\mathcal{J}(\bm{m}) \times \bm{H} (\bm{m})
			\right), 
		\end{aligned}
		&\text{in } \Omega \times [0,\infty),
		\\
		\nabla^{\mathfrak{h}}  \bm{m} \cdot\bm{n} = 0,  & \text{on } \partial \Omega \times [0,\infty),  
		\\
		\bm{m}(\cdot,0) = \bm{m}_{0},   &\text{in } \Omega,
	\end{cases}
\end{equation}
where $\mathcal{J}(\bm{m})$ is a perturbation of $\bm{m}$, given by
\begin{align}\label{definition of J}
	\mathcal{J}(\bm{m}) =\dfrac{\bm{m}}{\max\{|\bm{m}|,1\}}.
\end{align}
Note that if $|\bm{m}|=1$, we have
\[-\mathcal{J}(\bm{m})\times(\mathcal{J}(\bm{m})\times \cdot)={\rm Id}: T_{\bm{m}}\mathbb{S}^2\to T_{\bm{m}}\mathbb{S}^2,\]
where ${\rm Id}$ denotes the identity map. Consequently, when $\ep=0$, the approximate system \eqref{I-galerkin1} reduces to the original equation \eqref{system1}.

\subsection{Galerkin approximation and a priori estimates}\label{ss: Ga-Type-I}
We employ a Galerkin approximation scheme to establish the well-posedness of the perturbed problem \eqref{I-galerkin1}, that is
\begin{equation}\label{I_projection}
    \begin{cases}
       \begin{aligned}
            \partial_{t} \bm{m}_n^\varepsilon - \varepsilon \Delta^{\mathfrak{h}} \bm{m}_n^\varepsilon =&
        \gamma  \mathcal{P}_n
          \left( \mathcal{J}(\bm{m}_n^\varepsilon )\times  \left(\mathcal{J}(\bm{m}_n^\varepsilon)\times (\bm{v} \cdot \nabla^{\mathfrak{h}} \bm{m}_n^\varepsilon)
          \right)\right)
          \\
         &+\gamma \mathcal{P}_n
          \left( \mathcal{J}(\bm{m}_n^\varepsilon )\times  \left(\mathcal{J}\left(\bm{m}_n^\varepsilon)\times (\bm{v} \times \mathcal{J}(\bm{m}_n^\varepsilon)\right)\right)\right)
    \\
       &- \alpha \mathcal{P}_n \left(
            
                   \mathcal{J}(\bm{m}_n^\varepsilon) \times  
         \bm{H}(\bm{m}_n^\varepsilon) 
 \right)
 \\
     &- \beta  \mathcal{P}_n 
     \left( 
     \mathcal{J}(\bm{m}_n^\varepsilon)\times
      \left( 
          \mathcal{J}(\bm{m}_n^\varepsilon)\times  \bm{H}(\bm{m}_n^\varepsilon) 
      \right) \right), 
       \end{aligned}
       &\text{in } \Omega \times [0,\infty),
      \\
      \nabla^{\mathfrak{h}}  \bm{m}_n^\varepsilon \cdot\bm{n} = 0,   &\text{on } \partial \Omega \times [0,\infty),  
      \\
      \bm{m}_n^\varepsilon (\cdot,0)= \mathcal{P}_n(\bm{m}_0), &\text{in } \Omega.
    \end{cases}
\end{equation}

Now, we seek solutions to problem \eqref{I_projection} in the space  $\mathcal{H}_n$. Let $\bm{m}_n^\varepsilon(x,t)=\sum_{i=1}^n g_i^n(t)\bm{f}_i(x)$. Then the vector-valued function $\bm{g}^n(t)=(g_1^n(t),g_2^n(t),\cdots,g_n^n(t))$ satisfies the following ordinary differential equation
\begin{equation*}
	\begin{cases}
		\dfrac{{\rm d}\bm{g}^n}{{\rm d}t} = -\varepsilon S \bm{g}^n + \bm{F}(t, \bm{g}^n)
		\\
		\bm{g}^n(0)=\left(\left\langle \bm{m}_0,\bm{f}_1\right\rangle ,\cdots,\left\langle \bm{m}_0,\bm{f}_n\right\rangle \right)
	\end{cases}
\end{equation*}
where the matrix $S=((\lambda_i-1)\delta_{ij})_{n\times n}$, and $\bm{F}(t,\bm{g}^n)$ is locally Lipschitz in $\bm{g}^n$, since $\mathcal{J}(\bm{u})$ is locally Lipschitz on $\bm{u}$. Hence, there exists a unique solution $\bm{m}_n^\varepsilon$ to the problem \eqref{I_projection} on $\Omega\times[0,T_n)$ for some $T_n> 0$ according to the standard existence theory for ordinary differential equations.

To proceed, we establish uniform energy estimates for $\bm{m}_n^\varepsilon$ with respect to $n$.
\begin{lemma}
 Suppose $\bm{m}_0 \in L^2(\Omega)$, we have the basic energy estimate:
 \begin{align}\label{I-energy}
  \sup\limits_{0\leq t\leq T} \int_{\Omega}|\bm{m}_n^\varepsilon|^2 {\rm d}x + 2\varepsilon \int_0^T \int_{\Omega}|\nabla^{\mathfrak{h}}\bm{m}_n^\varepsilon|^2 {\rm d}x  {\rm d}t
   \leq \int_{\Omega} |\bm{m}_0|^2 {\rm d}x,
  \quad \forall T \in (0,T_n).
 \end{align}
  Moreover, this estimate implies that $T_n =+\infty$, and hence $\bm{m}_n^\varepsilon$ is a global solution.
\end{lemma}
\begin{proof}
Multiplying the equation in \eqref{I_projection} by $\bm{m}_n^\varepsilon$ and then integrating over $\Omega$ by parts, we obtain
 \begin{align}
 \frac{1}{2}\frac{{\rm d}}{{\rm d}t} \|\bm{m}_n^\varepsilon\|_{L^2(\Omega)}^2
 +\varepsilon \|\nabla^{\mathfrak{h}} \bm{m}_n^\varepsilon\|_{L^2(\Omega)}^2=0,\label{I-energy+1}
 \end{align}
where we have applied the property $(\mathcal{J}(\bm{m}_n^\varepsilon)\times \bm{b})\cdot\bm{m}_n^\varepsilon=0$ for any vector $\bm{b}$.  

Integrating \eqref{I-energy+1} from $0$ to $T$, we get the inequality \eqref{I-energy}. Moreover, noting that $\langle \bm{g}^n(t), \bm{g}^n(t)\rangle = \|\bm{m}_n^\varepsilon(t)\|_{L^2(\Omega)}^2$ for any $ t <T_n$, the estimate \eqref{I-energy} implies  
\begin{align*}
    \sup\limits_{t \in (0,T_n)} |\bm{g}^n(t)|^2 \leq \|\bm{m}_0\|_{L^2(\Omega)}^2,
\end{align*}
 which implies that $T_n$ can be extended  to $+\infty.$ The proof is completed.
\end{proof}

Next, we establish the following uniform $H^1$-estimates for $\bm{m}_n^\varepsilon$.
\begin{lemma}\label{lemma3.2}
Suppose that $\bm{m}_0 \in H^1(\Omega;\mathbb{S}^2)$, $\bm{v}\in L^2_{loc} ([0,\infty); L^\infty(\Omega))$, $\bm{f}\in L^2_{loc}([0,\infty);L^2(\Omega))$ and $\bm{\pi}$ satisfies \eqref{Cpi}, then for any $T>0$ the solution $\bm{m}_n^\varepsilon $  satisfies
\begin{equation}\label{I_H^1}
\frac{1}{2}  \sup\limits_{0\leq t\leq T}\| \bm{m}_n^\varepsilon(t)\|_{H_{\mathfrak{h}}^1(\Omega)}^2
+ 
\varepsilon \int_0^T\int_{\Omega}|\Delta^{\mathfrak{h}} \bm{m}_n^\varepsilon|^2 {\rm d} x{\rm d}t
+ \frac{\beta}{2} \int_0^T\int_{\Omega}  | \mathcal{J}(\bm{m}_n^\varepsilon)\times 
     \Delta^{\mathfrak{h}} \bm{m}_n^\varepsilon |^2{\rm d} x{\rm d}t
\leq \mathfrak{C}_1,
\end{equation}
where the constant $\mathfrak{C}_1$ depends only on $\al, \beta, \ga, T$, $\norm{\bm{v}}_{L^2([0,T]; L^\infty(\Om))}$ and $\norm{\bm{f}}_{L^2([0,T]\times \Om)}$.
\end{lemma}
\begin{proof}
Taking $-\Delta^{\mathfrak{h}} \bm{m}_n^\varepsilon$ as a test function for the equation \eqref{I_projection} and using  Proposition \ref{lemma2.1}, we have
\begin{equation}\label{es-H-1m_n}
\begin{aligned}
&\frac{1}{2} \frac{{\rm d}}{{\rm d}t} \|\nabla^{\mathfrak{h}} \bm{m}_n^\varepsilon\|_{L^2(\Omega)}^2
+ 
\varepsilon \|\Delta^{\mathfrak{h}} \bm{m}_n^\varepsilon\|_{L^2(\Omega)}^2 \\
=&-\gamma  \int_{\Omega}  
          \left(
              \mathcal{J}(\bm{m}_n^\varepsilon) 
              \times 
              \left( \mathcal{J}(\bm{m}_n^\varepsilon) \times \left(\bm{v} \cdot \nabla^{\mathfrak{h}} \bm{m}_n^\varepsilon +\bm{v} \times \mathcal{J}(\bm{m}_n^\varepsilon)\right)\right)
            \right)
          \cdot \Delta^{\mathfrak{h}}\bm{m}_n^\varepsilon 
          {\rm d}x
\\
&
+\alpha \int_{\Omega}  
\left(
\mathcal{J}(\bm{m}_n^\varepsilon)\times   \bm{H}(\bm{m}_n^\varepsilon)\right)\cdot \Delta^{\mathfrak{h}} \bm{m}_n^\varepsilon {\rm d} x
\\
&
+\beta\int_\Omega \left(\mathcal{J}(\bm{m}_n^\varepsilon) \times \left(\mathcal{J}(\bm{m}_n^\varepsilon) \times  \bm{H}(\bm{m}_n^\varepsilon)\right)\right)
\cdot \Delta^{\mathfrak{h}} \bm{m}_n^\varepsilon\:{\rm d}x\\
:=&A_1+A_2+A_3.
\end{aligned}
\end{equation}

The three terms $A_1$-$A_3$ admit the following estimates:
\begin{align*}
        A_1=&\gamma \int_{\Omega} 
            \left(
                     \mathcal{J}(\bm{m}_n^\varepsilon)  
                     \times 
                     \Delta^{\mathfrak{h}}\bm{m}_n^\varepsilon
            \right)
            \cdot 
             \left(
                     \mathcal{J}(\bm{m}_n^\varepsilon)  
                     \times 
                     \left(
                         \bm{v} \cdot \nabla^{\mathfrak{h}} \bm{m}_n^\varepsilon 
                     +
                        \bm{v} \times \mathcal{J}(\bm{m}_n^\varepsilon)
                    \right) 
            \right){\rm d}x
        \\
        \leq& \dfrac{\beta}{8} \int_{\Omega} 
         \left|
             \mathcal{J}(\bm{m}_n^\varepsilon)  
                     \times 
                     \Delta^{\mathfrak{h}}\bm{m}_n^\varepsilon
        \right|^2    {\rm d}x
        + 
        C\dfrac{\gamma^2}{\beta} 
        \| \bm{v}\|_{L^\infty(\Omega)}^2 
        \int_{\Omega}
        |\nabla^{\mathfrak{h}} \bm{m}_n^\varepsilon|^2
        + 
        |\bm{m}_n^\varepsilon|^2 
        {\rm d}x,\\
A_2=& 
    - \alpha \int_{\Omega}    
 \left(
     \mathcal{J}(\bm{m}_n^\varepsilon)\times 
     \Delta^{\mathfrak{h}} \bm{m}_n^\varepsilon 
\right)
 \cdot 
  \left(
         \bm{\pi}( \bm{m}_n^\varepsilon) + \bm{f}
         \right)
 {\rm d} x
 \\
 &\leq \frac{\beta}{8} \int_{\Omega}  | \mathcal{J}(\bm{m}_n^\varepsilon)\times 
     \Delta^{\mathfrak{h}} \bm{m}_n^\varepsilon |^2{\rm d} x
     + \frac{2\alpha^2}{\beta}\int_{\Omega} |\bm{\pi}( \bm{m}_n^\varepsilon)|^2 + |\bm{f}|^2 {\rm d}x,\\
A_3=&-\beta \int_\Omega\left(\mathcal{J}(\bm{m}_n^\varepsilon) \times 
    \Delta^{\mathfrak{h}} \bm{m}_n^\varepsilon\right)
    \cdot  \left(\mathcal{J}(\bm{m}_n^\varepsilon) \times 
    \left( 
    \Delta^{\mathfrak{h}} \bm{m}_n^\varepsilon
    +  \bm{\pi}( \bm{m}_n^\varepsilon) + \bm{f}
    \right)
    \right)
    \:{\rm d}x
\\
    &\leq -\frac{3\beta}{4} \int_{\Omega}  | \mathcal{J}(\bm{m}_n^\varepsilon)\times 
     \Delta^{\mathfrak{h}} \bm{m}_n^\varepsilon |^2{\rm d} x
     +\beta\int_{\Omega} |\bm{\pi}( \bm{m}_n^\varepsilon)|^2 +|\bm{f}|^2 {\rm d}x.
\end{align*}

Substituting the above estimates of $A_1$-$A_3$ into \eqref{es-H-1m_n}, we get
\begin{align*}
    &\quad\frac{1}{2} \frac{{\rm d}}{{\rm d}t} \|\nabla^{\mathfrak{h}} \bm{m}_n^\varepsilon\|_{L^2(\Omega)}^2
+ 
\varepsilon \|\Delta^{\mathfrak{h}} \bm{m}_n^\varepsilon\|_{L^2(\Omega)}^2
+ \frac{\beta}{2} 
\| \mathcal{J}(\bm{m}_n^\varepsilon)\times 
     \Delta^{\mathfrak{h}} \bm{m}_n^\varepsilon \|_{L^2(\Omega)}^2
\\
\leq 
&
   C\| \bm{v}\|_{L^\infty(\Omega)}^2    \norm{\bm{m}_n^\varepsilon}_{H^1_{\mathfrak{h}}(\Omega)}^2
     +C
\left(
     \norm{\bm{m}_n^\varepsilon}_{L^2(\Omega)}^2 
     +\norm{\bm{f} }_{L^2(\Omega)}^2
     \right).
\end{align*}
By Gronwall's inequality, this inequality implies the desired estimate \eqref{I_H^1}. Here we have applied the following bound
\begin{align*}
	\int_{\Omega}\left| \nabla^{\mathfrak{h}}  \mathcal{P}_n(\bm{m}_0)\right|^2{\rm d}x \leq \int_{\Omega} \left|\nabla^{\mathfrak{h}} \bm{m}_0\right|^2 {\rm d}x.
\end{align*}
The proof is completed.
\end{proof}

A direct corollary of Lemma \ref{lemma3.2} is the following estimate for $\partial_t \bm{m}_n^\varepsilon$.
\begin{lemma}\label{I_times estimate}
  Under the same assumptions of Lemma \ref{lemma3.2}, we have
    \begin{equation}
       \begin{aligned}\label{I-time_derivative}
\|\partial_{t} \boldsymbol{m}_{n}^{\varepsilon}\|_{L^{2}([0,T]; L^2(\Omega))} \leq 
 \mathfrak{C}_1.
\end{aligned}
\end{equation}
\end{lemma}

Consequently, we obtain the following result.
\begin{proposition}\label{prop3.4}
Suppose that $\bm{m}_0 \in H^1(\Omega;\mathbb{S}^2)$, $ \bm{v}\in L^2_{loc} ([0,\infty); L^\infty(\Omega))$, $\bm{f}\in L^2_{loc}([0,\infty);L^2(\Omega))$. Then the projected problem \eqref{I_projection} admits a global solution $\bm{m}_n^\varepsilon$ such that for any $0<T<\infty$, $\bm{m}_n^\varepsilon\in L^\infty([0,T];H^1(\Omega)) \cap L^2([0,T];H^2(\Omega))$ and $\partial_t \bm{m}_n^\varepsilon \in L^2([0,T];L^2(\Omega))$. Moreover,  $\bm{m}_n^\varepsilon$ satisfies the a priori estimates \eqref{I_H^1} and \eqref{I-time_derivative}.
\end{proposition} 

\subsection{Compactness of the approximate solutions and the proof of Theorem \ref{theroem1}}

By applying Lions-Aubin-Simon compactness  (i.e., Lemma \ref{LAS}), we show that the sequence of approximate solutions $\{\bm{m}^\varepsilon_n\}$ admits a subsequence that converges to a strong solution of \eqref{I-galerkin1} as $n\to \infty$. 

In view of Propositions \ref{proposition2.1} and \ref{proposition2}, Proposition \ref{prop3.4} implies that there exists a subsequence of $\{\bm{m}_n^\varepsilon\}$ (for simplicity, we still denote it by $\{\bm{m}_n^\varepsilon\}$) such that
\begin{align}
\bm{m}_n^\varepsilon &\rightharpoonup \bm{m}^\varepsilon, \quad \text{ weakly* in } L^\infty([0,T];H^1(\Omega)), \label{eq:weakconv1} \\
\bm{m}_n^\varepsilon &\rightharpoonup \bm{m}^\varepsilon, \quad \text{ weakly in } L^2([0,T];H^2(\Omega)), \label{eq:weakconv2}\\
\p_t\bm{m}_n^\varepsilon &\rightharpoonup \p_t\bm{m}^\varepsilon,\: \text{ weakly in } L^2([0,T];L^2(\Omega)). \label{eq:weakconv3}
\end{align}

Next, let $\mathcal{X} = H^2(\Omega)$, $\mathcal{B} =H^1(\Om) $ and $\mathcal{Y} =L^2(\Omega)$.  Then Lemma \ref{LAS} yields
\begin{equation}
\bm{m}_n^\varepsilon \to \bm{m}^\varepsilon, \text{ strongly in } L^p([0,T];H^1(\Omega)), \label{eq:strongconv}
\end{equation}
for any $1\leq p<\infty$.

\begin{proposition}\label{lemma3.5}
The limit $ \bm{m}^\varepsilon $ of the sequence $\{\bm{m}_n^\varepsilon\}$ is a strong solution to the problem \eqref{I-galerkin1}, which satisfies the same a priori estimates as $ \bm{m}_n^\varepsilon $ given in \eqref{I_H^1} and \eqref{I-time_derivative}.
\end{proposition}

\begin{proof}
For any vector function $\bm{\varphi}  \in C^\infty(\bar{\Omega} \times [0,T]) $, we denote $\bm{\varphi}_m=\P_m(\bm{\varphi})$. Then for any $n\geq m$, the approximate solution $ \bm{m}_n^\varepsilon $ satisfies
\begin{align*}
 \int_0^T \int_\Omega  \partial_t \bm{m}_n^\varepsilon\cdot\bm{\varphi}_m {\rm d}x{\rm d}t 
&= \gamma  \int_0^T \int_\Omega
          \left( 
              \mathcal{J}(\bm{m}_n^\varepsilon)
              \times ( \mathcal{J}(\bm{m}_n^\varepsilon) 
              \times 
              (\bm{v} \cdot \nabla^{\mathfrak{h}} \bm{m}_n^\varepsilon
             +\bm{v} \times \mathcal{J}(\bm{m}_n^\varepsilon)) )
      \right)\cdot \bm{\varphi}_m  {\rm d}x{\rm d}t
\\
&\quad -\alpha \int_0^T \int_\Omega (\mathcal{J}(\bm{m}_n^\varepsilon) \times \bm{H}(\bm{m}_n^\varepsilon) )\cdot \bm{\varphi}_m  {\rm d}x{\rm d}t 
+ \varepsilon \int_0^T \int_\Omega  \Delta^{\mathfrak{h}} \bm{m}_n^\varepsilon\cdot \bm{\varphi}_m  {\rm d}x{\rm d}t
\\
&\quad -
\beta \int_0^T \int_\Omega  \left(\mathcal{J}(\bm{m}_n^\varepsilon) \times   (\mathcal{J}(\bm{m}_n^\varepsilon) \times  \bm{H}(\bm{m}_n^\varepsilon) )\right)\cdot \bm{\varphi}_m  {\rm d}x{\rm d}t.
\end{align*}

Based on the convergences  \eqref{eq:weakconv1}-\eqref{eq:strongconv}, we can show the following convergences 
\begin{align*}
\partial_t \bm{m}_n^\varepsilon &\rightharpoonup \partial_t \bm{m}^\varepsilon  &&\text{ weakly  in } L^2([0,T];L^2(\Omega)),
\\
\Delta^{\mathfrak{h}} \bm{m}_n^\varepsilon &\rightharpoonup\Delta^{\mathfrak{h}} \bm{m}^\varepsilon  &&\text{ weakly in } L^2([0,T];L^2(\Omega)),
\\
\mathcal{J}(\bm{m}_n^\varepsilon) &\rightarrow \mathcal{J}(\bm{m}^\varepsilon)  &&\text{ a.e. } (x,t) \in \Omega \times [0,T],
\\
\mathcal{J}(\bm{m}_n^\varepsilon) \times  \Delta^{\mathfrak{h}} \bm{m}_n^\varepsilon &\rightharpoonup \mathcal{J}(\bm{m}^\varepsilon)\times\Delta^{\mathfrak{h}} \bm{m}^\varepsilon  &&\text{ weakly in } L^2([0,T];L^2(\Omega)).
\end{align*}
Moreover, since $\bm{v}\in L^2([0,T];L^\infty(\Om))$, we also derive that
\begin{align*}
\mathcal{J}(\bm{m}_n^\varepsilon) \times( \bm{v} \cdot \nabla^{\mathfrak{h}} \bm{m}_n^\varepsilon) &\rightharpoonup \mathcal{J}(\bm{m}^\varepsilon) \times( \bm{v} \cdot \nabla^{\mathfrak{h}} \bm{m}^\varepsilon)  &&\text{ weakly in } L^1([0,T];L^2(\Omega)), 
\\
\mathcal{J}(\bm{m}_n^\varepsilon) \times( \bm{v} \times \mathcal{J}(\bm{m}_n^\varepsilon)) &\rightharpoonup \mathcal{J}(\bm{m}^\varepsilon) \times( \bm{v} \times \mathcal{J}(\bm{m}^\varepsilon)) &&\text{ weakly in } L^2([0,T];L^2(\Omega)).
\end{align*}
By taking $ n\to \infty$ and then $m\to \infty$, we obtain the following identity
\begin{align*}
 \int_0^T \int_\Omega  \partial_t \bm{m}^\varepsilon\cdot\bm{\varphi}  {\rm d}x{\rm d}t 
&= \gamma  \int_0^T \int_\Omega  
          \left( 
              \mathcal{J}(\bm{m}^\varepsilon)
              \times ( \mathcal{J}(\bm{m}^\varepsilon) 
              \times 
              (\bm{v} \cdot \nabla^{\mathfrak{h}} \bm{m}^\varepsilon
             + \bm{v} \times \mathcal{J}(\bm{m}^\varepsilon)) )
      \right)\cdot \bm{\varphi}  {\rm d}x{\rm d}t
\\
&\quad -\alpha \int_0^T \int_\Omega  (\mathcal{J}(\bm{m}^\varepsilon) \times \bm{H}(\bm{m}^\varepsilon)) )\cdot \bm{\varphi}  {\rm d}x{\rm d}t 
+ \varepsilon \int_0^T \int_\Omega  \Delta^{\mathfrak{h}} \bm{m}^\varepsilon\cdot \bm{\varphi}  {\rm d}x{\rm d}t
\\
&\quad -
\beta \int_0^T \int_\Omega  \left(\mathcal{J}(\bm{m}^\varepsilon) \times   (\mathcal{J}(\bm{m}^\varepsilon) \times  \bm{H}(\bm{m}^\varepsilon) )\right)\cdot \bm{\varphi}  {\rm d}x{\rm d}t.
\end{align*}

Next, we check the  boundary condition.  For any $  \bm{\varphi} \in C^\infty(\bar{\Omega} \times [0,T]) $, we find 
\[
\int_0^T \int_\Omega  \Delta^{\mathfrak{h}} \bm{m}_n^\varepsilon\cdot \bm{\varphi}  {\rm d}x{\rm d}t = -\int_0^T \int_\Omega  \nabla^{\mathfrak{h}}  \bm{m}_n^\varepsilon\cdot  \nabla^{\mathfrak{h}}  \bm{\varphi} {\rm d}x{\rm d}t
.\]
This immediately yields, as $n \rightarrow +\infty$, 
\[
\int_0^T \int_\Omega  \Delta^{\mathfrak{h}}  \bm{m}^\varepsilon\cdot \bm{\varphi}  {\rm d}x{\rm d}t = -\int_0^T \int_\Omega   \nabla^{\mathfrak{h}}   \bm{m}^\varepsilon\cdot  \nabla^{\mathfrak{h}}   \bm{\varphi} {\rm d}x{\rm d}t,
\]
which  implies  that $ \nabla^{\mathfrak{h}}  \bm{m}^\varepsilon \cdot{ \bm{n}}|_{\partial \Omega \times [0,T]} = 0 $. Consequently, $ \bm{m}^\varepsilon $ is a strong solution to the problem \eqref{I-galerkin1}.

Finally, when $n\to \infty$, the lower semi-continuity of \eqref{I_H^1} and \eqref{I-time_derivative} yields the desired estimates of $\bm{m}^\ep$. 
The proof is completed.
\end{proof}

By applying a similar approach to that in \cite{Chen2024Incompressible}, we derive the following $L^\infty$-estimate of $\bm{m}^\ep$.
\begin{lemma}
\label{I-lemma:max_principle}
Let $\bm{m}^\varepsilon$ be the  solution given in Proposition  \ref{lemma3.5}, which satisfies  the following  approximate system
\begin{equation}\label{I-maximum1}
\begin{cases}
\begin{aligned}
\partial_t \bm{m}^\varepsilon -\varepsilon \Delta^{\mathfrak{h}} \bm{m}^\varepsilon
=&
    \gamma \left(\mathcal{J}(\bm{m}^\varepsilon) \times \left(\mathcal{J}(\bm{m}^\varepsilon)\times (\bm{v} \cdot \nabla^{\mathfrak{h}} \bm{m}^\varepsilon) \right)\right)
    \\
    &  + \gamma \left(\mathcal{J}(\bm{m}^\varepsilon) \times (\mathcal{J}(\bm{m}^\varepsilon) \times( \bm{v} \times \mathcal{J}(\bm{m}^\varepsilon)))\right)
           \\
      &
 -\alpha \mathcal{J}(\bm{m}^\varepsilon) \times  
           \bm{H}(\bm{m}^\varepsilon) 
         \\
         &- \beta  \mathcal{J}(\bm{m}^\varepsilon) \times
      \left( 
           \mathcal{J}(\bm{m}^\varepsilon) \times    \bm{H}(\bm{m}^\varepsilon) 
      \right),
\end{aligned}
& {\rm in } \: \Omega \times [0,\infty),
      \\
      \nabla^{\mathfrak{h}}  \bm{m}^\varepsilon \cdot\bm{n} = 0,  & { \rm on }\:  \partial \Omega \times [0,\infty),  
      \\
    \bm{m}^\varepsilon(\cdot,0)=\bm{m}_{0}, &{\rm in }\: \Omega.
\end{cases}
\end{equation}
Then $|\bm{m}^\varepsilon| \leq 1$ for almost every $(x,t) \in \Omega \times [0,T]$ and any $T < \infty$.
\end{lemma}

\begin{proof}
We consider the auxiliary function
\[
\Phi(t) = \frac{1}{2} \int_{\Om} ((|\bm{m}^\varepsilon| - 1)_+)^2 \, {\rm d}x,
\]
where $(|\bm{m}^\varepsilon| - 1)_+ := \max\{|\bm{m}^\varepsilon| - 1, 0 \}$.
Differentiating $\Phi(t)$ with respect to time  and employing the equation \eqref{I-maximum1} yield
\begin{align*}
    \frac{\mathrm{d} \Phi}{{\rm d}t} 
    &= 
    \int_\Omega  (|\bm{m}^\varepsilon| - 1)_+  
     \partial_t (|\bm{m}^\varepsilon| - 1)_+  \, {\rm d}x
 = \int_{\{|\bm{m}^\varepsilon| > 1\}} (|\bm{m}^\varepsilon| - 1) \partial_t |\bm{m}^\varepsilon| \, {\rm d}x
\\
&=\varepsilon \int_{\{|\bm{m}^\varepsilon| > 1\}} 
    \left(\frac{|\bm{m}^\varepsilon| - 1}{|\bm{m}^\varepsilon|}\right) \bm{m}^\varepsilon \cdot\Delta^\mathfrak{h} \bm{m}^\varepsilon 
{\rm d}x\\
&=-\int_{\{|\bm{m}^\varepsilon| > 1\}} \nabla^{\mathfrak{h}} \left( \frac{|\bm{m}^\varepsilon| - 1}{|\bm{m}^\varepsilon|} \bm{m}^\varepsilon \right) \cdot \nabla^{\mathfrak{h}} \bm{m}^\varepsilon  {\rm d}x
\\
&= - \int_{\{|\bm{m}^\varepsilon| > 1\}}
 \frac{|\bm{m}^\varepsilon| - 1}{|\bm{m}^\varepsilon|} |\nabla^{\mathfrak{h}} \bm{m}^\varepsilon|^2   -\sum_{i=1}^{3} \partial_i |\bm{m}^\varepsilon|^{-1}\left( \bm{m}^\varepsilon \cdot\partial_i^{\mathfrak{h}} \bm{m}^\varepsilon \right) {\rm d}x 
\\
&= - \int_{\{|\bm{m}^\varepsilon| > 1\}} 
  \frac{|\bm{m}^\varepsilon| - 1}{|\bm{m}^\varepsilon|} |\nabla^{\mathfrak{h}} \bm{m}^\varepsilon|^2  
  + 
  \sum_{i=1}^{3}  \left( |\bm{m}^\varepsilon|^{-3}\left|\bm{m}^\varepsilon \cdot\partial_i^{\mathfrak{h}} \bm{m}^\varepsilon\right|^2 \right){\rm d}x
\leq 0.
\end{align*}

Therefore, $\Phi(t)$ is a decreasing non-negative function. Since the initial data $|\bm{m}_0|=1$, i.e., $ \Phi(0)=0$, we have $\Phi (t) \equiv 0, \forall t >0$, namely $|\bm{m}^\varepsilon|\leq 1$ a.e. in $\Omega \times [0,T]$. The proof is completed.
\end{proof}

Moreover, the bound $|\bm{m}^\varepsilon|\leq 1$ implies $ \mathcal{J}(\bm{m}^\varepsilon) = \bm{m}^\varepsilon $. Hence, $\bm{m}^\ep$ solves the following system
\begin{equation}\label{I-approxi_2}
\begin{cases}
\begin{aligned}
    \partial_t \bm{m}^\varepsilon -\varepsilon \Delta^{\mathfrak{h}} \bm{m}^\varepsilon 
&= 
    \gamma
    \left(\bm{m}^\varepsilon \times (\bm{m}^\varepsilon \times (\bm{v} \cdot \nabla^{\mathfrak{h}} \bm{m}^\varepsilon +\bm{v} \times  \bm{m}^\varepsilon) )\right)
      \\
      &\quad -\alpha \bm{m}^\varepsilon \times  
           \bm{H}(\bm{m}^\varepsilon) 
           - \beta  \bm{m}^\varepsilon\times
      \left( 
           \bm{m}^\varepsilon \times    \bm{H}(\bm{m}^\varepsilon)
      \right),
\end{aligned} & \text{in } \Omega \times [0,\infty),
      \\
      \nabla^{\mathfrak{h}}  \bm{m}^\varepsilon \cdot\bm{n} = 0,  & \text{on } \partial \Omega \times [0,\infty),  
      \\
    \bm{m}^\varepsilon(\cdot,0) = \bm{m}_{0}, &\text{in }\Omega.
\end{cases}
\end{equation}

With the aid of the above results, we now prove Theorem \ref{theroem1}.

\begin{proof}[\textbf{Proof of Theorem \ref{theroem1}.}]
The proof is divided into two steps.

\noindent\textbf{Step 1: The convergence of $\bm{m}^\varepsilon$.} For any $T < \infty$, Proposition \ref{lemma3.5} implies that
\[
\|\bm{m}^\varepsilon\|_{L^{\infty}([0,T];H^{1}(\Omega))} +\|\bm{m}^\varepsilon\times 
     \Delta^{\mathfrak{h}} \bm{m}^\varepsilon \|_{L^{2}([0,T];L^2(\Omega))} 
+ \|\partial_t \bm{m}^\varepsilon\|_{L^{2}([0,T];L^2(\Omega))} \leq \mathfrak{C}_1,
\]
for some constant $\mathfrak{C}_1$ independent of $\varepsilon$.
Then, there exist a subsequence of $\{\bm{m}^\ep\}$ (still denoted by $\{\bm{m}^\ep\}$ for simplicity) and a map $\bm{m} \in L^{\infty}([0,T];L^{2}(\Omega))$ such that
\begin{align*}
\bm{m}^\varepsilon &\rightharpoonup \bm{m}, &&\text{weakly* in } L^{\infty}([0,T];H^{1}(\Omega)),
 \\
\partial_t  \bm{m}^\varepsilon &\rightharpoonup \partial_t  \bm{m}, &&\text{weakly in } L^{2}([0,T];L^{2}(\Omega)),
 \\
\nabla^\mathfrak{h}  \bm{m}^\varepsilon &\rightharpoonup \nabla^\mathfrak{h}  \bm{m}, &&\text{weakly in } L^{2}([0,T];L^{2}(\Omega)),\\
\bm{m}^\varepsilon\times 
\nabla^{\mathfrak{h}} \bm{m}^\varepsilon
&\rightharpoonup  
\bm{m} \times 
\nabla^{\mathfrak{h}} \bm{m},
&&\text{weakly in } L^{2}([0,T];L^{2}(\Omega)),
\end{align*}
as $\ep\to 0$. Here we have used the bound $|\bm{m}^\ep|\leq 1$.

Let $\mathcal{X} = H^{1}(\Omega)$, $\mathcal{B} =L^{2}(\Omega)$, $\mathcal{Y} =L^2(\Omega)$. We conclude from Lemma \ref{LAS}  that 
\[
\bm{m}^\varepsilon \to \bm{m}, \quad \text{strongly in } C^{0}([0,T];L^2(\Omega)), \text{ as } \varepsilon \to 0.
\]

Furthermore, we derive  $|\bm{m}| = 1$. Precisely, by multiplying the equation in  \eqref{I-approxi_2} by the test function $\bm{m}^\varepsilon$   and using the fact that $|\bm{m}_0| = 1$, we obtain
\[
\int_{\Omega} |\bm{m}^\varepsilon|^{2}{\rm d}x + 2\varepsilon \int_{0}^{T} \int_{\Omega} |\nabla^{\mathfrak{h}} \bm{m}^\varepsilon|^{2}{\rm d}x{\rm d}t = \int_{\Omega} |\bm{m}_0|^{2}{\rm d}x,
\]
for a.e. $t \in [0,T]$. As $\varepsilon \to 0$, there holds that
\[
\int_{\Omega} (|\bm{m}|^{2} - 1){\rm d}x = 0.
\]
Since $|\bm{m}|\leq 1$, this implies $|\bm{m}| = 1$ for a.e. $(x,t) \in \Omega \times [0,T]$.

\medskip
\noindent\textbf{Step 2: Global existence of weak solutions to \eqref{system1}.} For any $\bm{\varphi} \in C^{\infty}(\bar{\Omega} \times [0,T])$, the solution $\bm{m}^\varepsilon$ satisfies the following equation
\begin{align*}
 &\int_{0}^{T} \int_{\Omega} \left( \partial_t \bm{m}^\varepsilon\cdot \bm{\varphi} \right) \:{\rm d}x{\rm d}t  
-\gamma \int_{0}^{T} \int_{\Omega} 
\bm{m}^\varepsilon\times(\bm{m}^\varepsilon\times (\bm{v} \cdot \nabla^{\mathfrak{h}} \bm{m}^\varepsilon+\bm{v} \times  \bm{m}^\varepsilon))
\cdot \bm{\varphi} \:{\rm d}x{\rm d}t\\
=& \varepsilon \int_{0}^{T} \int_{\Omega} \left( \Delta^\mathfrak{h} \bm{m}^\varepsilon\cdot \bm{\varphi} \right)\: {\rm d}x{\rm d}t
- \beta \int_{0}^{T} \int_{\Omega} \left( \bm{m}^\varepsilon \times ( \bm{m}^\varepsilon \times (\bm{\pi}(\bm{m}^\varepsilon)+ \bm{f}) )\right)\cdot \bm{\varphi} \:{\rm d}x{\rm d}t
\\
&- \beta \int_{0}^{T} \int_{\Omega} \left( \bm{m}^\varepsilon \times ( \bm{m}^\varepsilon \times \Delta^\mathfrak{h} \bm{m}^\varepsilon)\right)\cdot \bm{\varphi} \:{\rm d}x{\rm d}t
- \alpha \int_{0}^{T} \int_{\Omega} \left( \bm{m}^\varepsilon \times \Delta^\mathfrak{h} \bm{m}^\varepsilon\right)\cdot \bm{\varphi}\:{\rm d}x{\rm d}t
\\
& 
- \alpha \int_{0}^{T} \int_{\Omega} \left( \bm{m}^\varepsilon \times (\bm{\pi}(\bm{m}^\varepsilon)+ \bm{f}) \right)\cdot \bm{\varphi}\: {\rm d}x{\rm d}t .
\end{align*}

By passing $\varepsilon \to 0$, we apply the convergence results for $\bm{m}^\varepsilon$ in Step 1 to show
\begin{align*}
\int_{0}^{T} \int_{\Omega} \left( \partial_t \bm{m}^\varepsilon\cdot \bm{\varphi} \right) {\rm d}x{\rm d}t\to& 
 \int_{0}^{T} \int_{\Omega} \left( \partial_t \bm{m}\cdot \bm{\varphi} \right) \:{\rm d}x{\rm d}t,\\
\varepsilon \int_{0}^{T} \int_{\Omega} \left( \Delta^\mathfrak{h} \bm{m}^\varepsilon\cdot \bm{\varphi} \right) {\rm d}x{\rm d}t 
=& -\varepsilon \int_{0}^{T} \int_{\Omega} \left( \nabla^\mathfrak{h} \bm{m}^\varepsilon\cdot \nabla \bm{\varphi} \right)\:{\rm d}x{\rm d}t \to 0,
\\
 -\gamma \int_{0}^{T} \int_{\Omega} 
 \bm{m}^\varepsilon\times(\bm{m}^\varepsilon\times (\bm{v} \cdot \nabla^{\mathfrak{h}} \bm{m}^\varepsilon+\bm{v} \times  \bm{m}^\varepsilon))
 \cdot \bm{\varphi}\:{\rm d}x{\rm d}t
        \to &
         \:\gamma\int_{0}^{T} \int_{\Omega} 
              (\bm{v} \cdot \nabla^{\mathfrak{h}} \bm{m} +\bm{v} \times  \bm{m} )\cdot \bm{\varphi} \:{\rm d}x{\rm d}t,
\end{align*}
and 
\begin{align*}
-\alpha \int_{0}^{T} \int_{\Omega} \left( \bm{m}^\varepsilon \times \Delta^\mathfrak{h} \bm{m}^\varepsilon\right)\cdot \bm{\varphi}\: {\rm d}x{\rm d}t =& \alpha\int_{0}^{T} \int_{\Omega} \left( \bm{m}^\varepsilon \times \nabla^\mathfrak{h} \bm{m}^\varepsilon\right)\cdot\nabla^\mathfrak{h} \bm{\varphi}\:{\rm d}x{\rm d}t 
\\
\to &\alpha\int_{0}^{T} \int_{\Omega} \left( \bm{m} \times \nabla^\mathfrak{h} \bm{m}\right) \cdot \nabla^\mathfrak{h} \bm{\varphi} \:{\rm d}x{\rm d}t.
\end{align*}

On the other hand, since $\bm{\pi}: L^2(\Om)\to L^2(\Om)$ is a bounded self-adjoint operator and $\bm{f}\in L^2(\Om\times [0,T])$, when $\ep\to 0$, we also have
\begin{align*}
     -\alpha\int_{0}^{T} \int_{\Omega} \left( \bm{m}^\varepsilon \times (\bm{\pi}(\bm{m}^\varepsilon)+ \bm{f}) \right)\cdot \bm{\varphi} \:{\rm d}x{\rm d}t
     \to & 
     -\alpha\int_{0}^{T} \int_{\Omega} \left( \bm{m} \times (\bm{\pi}(\bm{m})+ \bm{f}) \right)\cdot \bm{\varphi}\: {\rm d}x{\rm d}t,
     \end{align*}
     and 
     \begin{align*}
    -\beta  \int_{0}^{T} \int_{\Omega} \left( \bm{m}^\varepsilon \times ( \bm{m}^\varepsilon \times (\bm{\pi}(\bm{m}^\varepsilon)+ \bm{f}) )\right)\cdot \bm{\varphi} \:{\rm d}x{\rm d}t
    &=
    \beta  \int_{0}^{T} \int_{\Omega} \left( \bm{m}^\varepsilon \times (\bm{\pi}(\bm{m}^\varepsilon)+ \bm{f}) \right)\cdot (\bm{m}^\varepsilon \times \bm{\varphi})\: {\rm d}x{\rm d}t
    \\
    &\to  
    \beta\int_{0}^{T} \int_{\Omega} \left(  \bm{m} \times (\bm{\pi}(\bm{m})+ \bm{f}) \right)\cdot  (\bm{m} \times\bm{\varphi} )\:{\rm d}x{\rm d}t.
\end{align*}

It remains to establish the convergence of the following term
\begin{align*}
\quad -\beta\int_{0}^{T} \int_{\Omega} \left(\bm{m}^\varepsilon\times ( \bm{m}^\varepsilon \times \Delta^\mathfrak{h} \bm{m}^\varepsilon)\right) \cdot \bm{\varphi} {\rm d}x{\rm d}t 
	=&\beta\int_{0}^{T} \int_{\Omega} 
	\left(\bm{m}^\varepsilon \times \bm{\varphi}
	\right)
	\cdot
	(\bm{m}^\varepsilon \times \De^\mathfrak{h} \bm{m}^\varepsilon)
	{\rm d}x{\rm d}t\\
	=&-\beta\int_{0}^{T} \int_{\Omega} 
\n^\mathfrak{h}\left(\bm{m}\times \bm{\varphi}
	\right)
	\cdot
	(\bm{m}^\varepsilon \times \n^\mathfrak{h} \bm{m}^\varepsilon)
	{\rm d}x{\rm d}t+\mathcal{O}(\varepsilon)\\
	\to& -\beta\int_{0}^{T} \int_{\Omega} 
	\n^\mathfrak{h}\left(\bm{m}\times \bm{\varphi}
	\right)
	\cdot
	(\bm{m} \times \n^\mathfrak{h} \bm{m})
	\:{\rm d}x{\rm d}t,
\end{align*}
where we have applied the following estimate:
\begin{align*}
\left|\int_{0}^{T} \int_{\Omega} 
\left(\left(\bm{m}^\varepsilon \times \bm{\varphi}
\right)-\left(\bm{m}\times \bm{\varphi}\right)\right)
\cdot
(\bm{m}^\varepsilon \times \De^\mathfrak{h} \bm{m}^\varepsilon)
\:{\rm d}x{\rm d}t \right|
\leq &\mathfrak{C}_1\norm{\bm{\varphi}}_{L^\infty(\Om)}\norm{\bm{m}^\varepsilon-\bm{m}}_{C^{0}([0,T];L^2(\Omega))}\\
=&\mathcal{O}(\varepsilon)\to 0.
\end{align*}

To summarize the above convergence results, we conclude that the limiting map $\bm{m}$ satisfies the  equation
\begin{equation}\label{weakI.1}
\begin{aligned}
&\int_{0}^{T} \int_{\Omega} 
D_t^\mathfrak{h} \bm{m} \cdot \bm{\varphi}
\:{\rm d}x{\rm d}t\\
=&\int_{0}^{T} \int_{\Omega}  \left( \partial_t\bm{m}\cdot  \bm{\varphi} \right) \:{\rm d}x{\rm d}t 
+
 \gamma\int_{0}^{T} \int_{\Omega} 
    \left(
             \bm{v} \cdot \nabla^{\mathfrak{h}} \bm{m} +(\bm{v} \times  \bm{m} )
            \right)\cdot \bm{\varphi} \:{\rm d}x{\rm d}t\\
=&\alpha\int_{0}^{T} \int_{\Omega} \left(( \bm{m} \times \nabla^\mathfrak{h}  \bm{m})\cdot \nabla^\mathfrak{h}  \bm{\varphi} \right) \:{\rm d}x{\rm d}t
        +\beta\int_{0}^{T} \int_{\Omega} \left(  \bm{m} \times (\bm{\pi}(\bm{m})+ \bm{f}) \right)\cdot  (\bm{m} \times\bm{\varphi} )\:{\rm d}x{\rm d}t 
\\
&-\beta\int_{0}^{T} \int_{\Omega} 
        \nabla^\mathfrak{h}\left(\bm{m} \times \bm{\varphi}
            \right)
            \cdot
             (\bm{m} \times \nabla^\mathfrak{h} \bm{m})
             \:{\rm d}x{\rm d}t 
             -\alpha\int_{0}^{T} \int_{\Omega} \left( \bm{m} \times (\bm{\pi}(\bm{m})+ \bm{f}) \right)\cdot \bm{\varphi}\: {\rm d}x{\rm d}t
\end{aligned}
\end{equation}
for any $\bm{\varphi} \in C^{\infty}(\bar{\Omega} \times [0,T])$ and any $T < \infty$.

Replacing $\bm{ \varphi}$ by $ -\bm{m} \times \bm{ \varphi}$ in the above equation \eqref{weakI.1}, we obtain 
\begin{equation}\label{weakI.2}
\begin{aligned}
	&\quad 
	\int_{0}^{T} \int_{\Omega}  (\bm{m} \times D_t^\mathfrak{h}\bm{m})\cdot  \bm{\varphi} \:{\rm d}x{\rm d}t 
	\\
	&= 
    \alpha\int_{0}^{T} \int_{\Omega} \left( \bm{m} \times (\bm{\pi}(\bm{m})+ \bm{f}) \right)\cdot ( \bm{m} \times\bm{\varphi} ){\rm d}x{\rm d}t
	-\alpha\int_{0}^{T} \int_{\Omega} ( \bm{m} \times \nabla^\mathfrak{h}  \bm{m})\cdot\nabla^\mathfrak{h}  (\bm{m} \times \bm{ \varphi} ){\rm d}x{\rm d}t
	\\
	&\quad +\beta\int_{0}^{T} \int_{\Omega} \left(  \bm{m} \times (\bm{\pi}(\bm{m})+ \bm{f}) \right)\cdot  \bm{\varphi} \:{\rm d}x{\rm d}t
    -\beta\int_{0}^{T} \int_{\Omega} 
	\nabla^\mathfrak{h}\bm{\varphi}
	\cdot
	(\bm{m} \times \nabla^\mathfrak{h} \bm{m})\:
	{\rm d}x{\rm d}t.
\end{aligned}
\end{equation}

 Taking $- \beta \times \eqref{weakI.2}$ and then adding $\alpha \times \eqref{weakI.1}$, we obtain that
\begin{align*}
&\alpha\int_0^T \int_{\Omega}  D^{\mathfrak{h}}_t\bm{m}\cdot\bm{\varphi}\:{\rm d } x {\rm d }t
-\beta\int_0^T \int_{\Omega} \left( \bm{m}\times D^{\mathfrak{h}}_t\bm{m}\right)\cdot\bm{\varphi}\:{\rm d } x {\rm d }t\\
=&(\alpha^2+\beta^2)\int_0^T \int_{\Omega}(\bm{m}\times \nabla^\mathfrak{h} \bm{m})\cdot\nabla^\mathfrak{h} \bm{\varphi}\: {\rm d}x{\rm d}t
		-(\alpha^2+\beta^2) \int_0^T \int_{\Omega} (\bm{m}\times (\bm{\pi}(\bm{m})+ \bm{f}))\cdot  \bm{\varphi} \:{\rm d}x{\rm d}t.
\end{align*}

On the other hand, since $\bm{m}^\ep\to \bm{m}$ strongly in $C^0([0,T]; L^2(\Om))$, it follows that
$$\bm{m}(\cdot, t) \to \bm{m}_0,\quad \text{in}\,\, L^2(\Om),\,\ \text{as}\,\, t\to 0.$$ 
Furthermore, by arguments similar to those used in Proposition~\ref{lemma3.5}, we also obtain $\nabla^\mathfrak{h} \bm{m} \cdot \bm{n}=0$
in the sense of traces.

Therefore,  the proof is completed.
\end{proof}

\medskip
In the end of this section, we establish an energy inequality for the weak solution given in Theorem \ref{theroem1}. 
\begin{proposition}
 Suppose that $ \bm{m}$ is a solution obtained in Theorem \ref{theroem1}. Then there exists a constant $ \delta>0$ such that, for every $T>0$, the following inequality holds:
 \begin{equation}\label{energy-ineq}
\begin{aligned}
   &\quad   \mathcal{E}^{\mathfrak{h}}[\bm{m}(T)]
   +\int_0^T\int_{\Omega}\bm{m}\cdot \partial_t \bm{f}\:{\rm d}x
  {\rm d}t
  + 
   \frac{\beta -\gamma \delta}{(\alpha^2+ \beta^2)(1+\gamma \delta)} \int_0^T\int_{\Omega}\left|\partial_t\bm{m}\right|^2\:{\rm d}x{\rm d}t
\\
&\leq 
\mathcal{E}^{\mathfrak{h}}[\bm{m}_0]
+
\left(\gamma^2 + \frac{\gamma}{\delta}\right)(1+\beta -\gamma \delta)
 \int_0^T\int_{\Omega}|\bm{v} \cdot \nabla^{\mathfrak{h}} \bm{m}  +\bm{v} \times  \bm{m} |^2\:{\rm d}x{\rm d}t,
\end{aligned}
\end{equation}
where $ \mathcal{E}^{\mathfrak{h}}[\bm{m}(t)]$ is defined in \eqref{helical_total_energy}.
\end{proposition}
\begin{proof}
For convenience, we denote 
$
\mathfrak{T}_{\bm{v}}:=  \bm{m}^\varepsilon \times (\bm{v} \cdot \nabla^{\mathfrak{h}} \bm{m}^\varepsilon + \bm{v} \times  \bm{m}^\varepsilon)
$ 
and
$ 
\mathfrak{T}_{\bm{H}}:= \bm{m}^\varepsilon \times \bm{H}(\bm{m}^\varepsilon)
$. 
We then multiply the equation in \eqref{I-approxi_2} by $ -\bm{H}(\bm{m}^\varepsilon) $ to obtain
\begin{align*}
\frac{{\rm d}}{{\rm d}t}\mathcal{E}^{\mathfrak{h}}[\bm{m}^\varepsilon(t)]+ 
\int_{\Omega}\bm{m}^\varepsilon\cdot \partial_t \bm{f}\:{\rm d}x=&-\int_{\Omega} \partial_t \bm{m}^\varepsilon  \cdot \bm{H}(\bm{m}^\varepsilon) \: {\rm d}x=B_1+B_2, 
\end{align*}
 where
  \begin{align*}
  B_1=&-\ep\int_{\Omega} 
  \De \bm{m}^\varepsilon 
  \cdot 
  \bm{H}(\bm{m}^\varepsilon)
  \:{\rm d}x
  =-\varepsilon\int_{\Omega} |\Delta^{\mathfrak{h}}\bm{m}^\varepsilon|^2+  \Delta^{\mathfrak{h}} \bm{m}^\varepsilon  \cdot ( \bm{\pi}(\bm{m^\varepsilon}) + \bm{f})   \:{\rm d}x\\
  \leq& \frac{\varepsilon }{2}\int_{\Omega} 
       |\bm{\pi}(\bm{m}^\varepsilon)|^2 + | \bm{f}|^2
      \:  {\rm d}x
 = \mathcal{O}(\varepsilon)
 \end{align*}
and 
\begin{align*}
 B_2=&-\gamma  \int_{\Omega} 
(\bm{m}^\varepsilon \times \mathfrak{T}_{\bm{v}})
 \cdot 
 \bm{H}(\bm{m}^\varepsilon)
 \:{\rm d}x 
 +
  \beta \int_{\Omega}  ( \bm{m}^\varepsilon \times \mathfrak{T}_{H} )  \cdot \bm{H}(\bm{m}^\varepsilon )\:{\rm d}x\\
 \leq&
  \frac{\gamma}{4\delta}  \int_{\Omega} 
  |\mathfrak{T}_{\bm{v}}|^2 {\rm d}x 
  +
  (\gamma\delta-\beta)\int_{\Omega}  |\mathfrak{T}_{\bm{H}}|^2 \: {\rm d}x.
  \end{align*}
This implies that
  \begin{align}\label{II-estimate_4}
       \frac{{\rm d}}{{\rm d}t} \mathcal{E}^{\mathfrak{h}}[\bm{m}^\varepsilon(t)]
       +\int_{\Omega}\bm{m}^\varepsilon\cdot \partial_t \bm{f}{\rm d}x
       +
        \mathcal{O}(\varepsilon) 
        +(\beta -\gamma \delta ) \|\mathfrak{T}_{\bm{H}}\|_{L^2(\Omega)}^2 
        \leq  
       \frac{\gamma}{4\delta}  \|\mathfrak{T}_{\bm{v}}\|_{L^2(\Omega)}^2,
  \end{align}
where $\int_{0}^{T}\mathcal{O}(\varepsilon)dt\to 0 $ as $\varepsilon\to0$. 

On the other hand, choosing $\partial_t \bm{m}^\varepsilon$ as a test function to \eqref{I-approxi_2}, we get 
\begin{equation}
      \begin{aligned}\label{II_estimate5}
      \| \partial_t \bm{m}^\varepsilon \|_{L^2(\Omega)}^2 
      -\frac{\varepsilon}{2} \frac{{\rm d}}{{\rm d}t}
       \|\nabla^\mathfrak{h} \bm{m}^\varepsilon \|_{L^2(\Omega)}^2 
 =&
       \int_{\Omega} 
 \left(
        \gamma\bm{m}^\varepsilon \times \mathfrak{T}_{\bm{v}}
         - \alpha \mathfrak{T}_{\bm{H}}
        -\beta  \bm{m}^\varepsilon \times  \mathfrak{T}_{\bm{H}}
 \right)
 \cdot 
 \partial_t \bm{m}^\varepsilon{\rm d}x\\
 =&C_1+C_2+C_2.
  \end{aligned}
\end{equation}
Next, we now turn  to  the terms $C_1$ to $C_3$ as follows. Noting that $|\bm{m}^\varepsilon|\leq 1$ and 
\begin{align*}
\partial_t \bm{m}^\varepsilon -\varepsilon \Delta^{\mathfrak{h}} \bm{m}^\varepsilon = 
\gamma\bm{m}^\varepsilon \times  \mathfrak{T}_{\bm{v}}-\alpha \mathfrak{T}_{\bm{H}}
- \beta  \bm{m}^\varepsilon\times
\mathfrak{T}_{\bm{H}},
\end{align*}
we obtain the following estimates
\begin{align*}
   C_1=&\gamma  \int_{\Omega} 
 \left(\bm{m}^\varepsilon \times \mathfrak{T}_{\bm{v}}\right)
 \cdot 
 \partial_t \bm{m}^\varepsilon{\rm d}x\\
\leq &
\varepsilon\gamma \int_{\Omega}  |\mathfrak{T}_{\bm{v}}||\mathfrak{T}_{\bm{H}}|{\rm d}x 
+ 
\gamma^2 \int_{\Omega} |\mathfrak{T}_{\bm{v}}|^2 {\rm d}x 
+ 
 (\alpha + \beta)\gamma  \int_{\Omega} |\mathfrak{T}_{\bm{v}}||\mathfrak{T}_{\bm{H}}|
 {\rm d}x 
 \\
 \leq& \mathcal{O}(\varepsilon)
 + \left(\gamma^2 + \frac{\gamma}{4\delta}\right)\|\mathfrak{T}_{\bm{v}}\|_{L^2(\Omega)}^2
 + \delta \gamma (\alpha^2+ \beta^2)\|\mathfrak{T}_{\bm{H}}\|_{L^2(\Omega)}^2,\\
C_2=&- \alpha\int_{\Omega}
    \mathfrak{T}_{\bm{H}}\cdot \partial_t \bm{m}^\varepsilon{\rm d}x\\
=& -\alpha\int_{\Omega}
     \varepsilon\mathfrak{T}_{\bm{H}} \cdot ( \mathfrak{T}_{\bm{H}}-\bm{\pi}(\bm{m^\varepsilon}) - \bm{f})
     -
     \alpha  |\mathfrak{T}_{\bm{H}}|^2
     +
     \gamma  \mathfrak{T}_{\bm{H}} \cdot (\bm{m}^\varepsilon \times\mathfrak{T}_{\bm{v}})
{\rm d}x
\\
\leq& \mathcal{O}(\varepsilon)
+  \alpha^2  (1+ \gamma \delta)\|\mathfrak{T}_{\bm{H}}\|_{L^2(\Omega)}^2
+\frac{\gamma}{4\delta}\|\mathfrak{T}_{\bm{v}}\|_{L^2(\Omega)}^2,
\end{align*}
and 
\begin{align*}
C_3=&    - \beta \int_{\Omega}  ( \bm{m}^\varepsilon \times  \mathfrak{T}_{\bm{H}})\cdot \partial_t \bm{m}^\varepsilon{\rm d}x
     \\
     =&
     -\beta \int_{\Omega} \varepsilon( \bm{m}^\varepsilon \times  \mathfrak{T}_{\bm{H}})\cdot (-\bm{\pi}(\bm{m^\varepsilon}) - \bm{f}))
     -
     \beta|\bm{m}^\varepsilon \times  \mathfrak{T}_{\bm{H}}|^2
     +
     \gamma (\bm{m}^\varepsilon \times  \mathfrak{T}_{\bm{H}})\cdot 
     (\bm{m}^\varepsilon \times  \mathfrak{T}_{\bm{v}}){\rm d}x
     \\
     \leq& \mathcal{O}(\varepsilon)
+  \beta^2 (1+\gamma \delta)\|\mathfrak{T}_{\bm{H}}\|_{L^2(\Omega)}^2
+ \frac{\gamma}{4\delta}\|\mathfrak{T}_{\bm{v}}\|_{L^2(\Omega)}^2.
\end{align*}

Substituting the above estimates of $C_1$-$C_3$ into \eqref{II_estimate5}, we get
\begin{align*}
 \|\mathfrak{T}_{\bm{H}}\|_{L^2(\Omega)}^2+\left(\gamma^2 + \frac{\gamma}{\delta}\right)\|\mathfrak{T}_{\bm{v}}\|_{L^2(\Omega)}^2 
    \geq& 
     \mathcal{O}(\varepsilon)  
     + \frac{1}{(\alpha^2+ \beta^2)(1+\gamma \delta)}\left(\| \partial_t \bm{m}^\varepsilon \|_{L^2(\Omega)}^2 
      -\frac{\varepsilon}{2} \frac{{\rm d}}{{\rm d}t}
       \|\nabla^\mathfrak{h} \bm{m}^\varepsilon \|_{L^2(\Omega)}^2 
       \right)
\end{align*}
which, together with  \eqref{II-estimate_4},  implies
\begin{align*}
     &\frac{{\rm d}}{{\rm d}t} \mathcal{E}^{\mathfrak{h}}[\bm{m}^\varepsilon(t)]
       +
       \int_{\Omega}\bm{m}^\varepsilon\cdot \partial_t \bm{f}{\rm d}x
       +
        \frac{\beta -\gamma \delta}{(\alpha^2+ \beta^2)(1+\gamma \delta)}\left(\| \partial_t \bm{m}^\varepsilon \|_{L^2(\Omega)}^2 
      -\frac{\varepsilon}{2} \frac{{\rm d}}{{\rm d}t}
       \|\nabla^\mathfrak{h} \bm{m}^\varepsilon \|_{L^2(\Omega)}^2 
       \right)
\\
\leq& 
     \mathcal{O}(\varepsilon)  
    +
    \left(\gamma^2 + \frac{\gamma}{\delta}\right)(1+\beta -\gamma \delta)\|\mathfrak{T}_{\bm{v}}\|_{L^2(\Omega)}^2.
\end{align*}
Integrating the above inequality over $[0,T]$ and letting $\varepsilon\to 0$, we obtain the desired energy inequality. 

The proof is completed.
\end{proof}

\section{Type-II weak solutions to LLG equation}\label{s: Type-II weak solution}
In this section, we prove Theorem \ref{theroem1'} and Theorem \ref{theorem2}. We assume that the vector $\bm{v}$ satisfies one of the following two regularity conditions:
\begin{itemize}
	\item[\eqref{(a)}:]  $ \bm{v} \in  L^2_{loc} ([0,\infty); L^3(\Omega)\cap W^{1,1}(\Omega))$, $\bm{v}\cdot \bm{n}|_{\partial \Omega \times  [0, \infty)}=0$ and $  \nabla \bm{v } =-( \nabla \bm{v })^T\,\,\text{in} \,\, \Omega\times [0, \infty)$;
	\item[\eqref{(c)}:] $\bm{v} \in  L^2_{loc}([0,\infty); L^3(\Omega)), \:  \nabla \bm{v} \in  L^1_{loc} ([0,\infty); L^\infty(\Omega)), \: \bm{v}\cdot \bm{n}|_{\partial \Omega \times  [0, \infty)}=0 \:\:{\rm and } \:\: {\rm div} \bm{v}=0 \,\,\text{in} \,\, \Omega\times [0, \infty)$.
\end{itemize} 
In this situation, we consider another approximate system of \eqref{system1'}:
\begin{equation}\label{II-galerkin1}
    \begin{cases}
       \begin{aligned}
            \partial_{t} \bm{m} - \varepsilon\Delta^{\mathfrak{h}} \bm{m}
      &=-\gamma\left(\bm{v} \cdot \nabla^{\mathfrak{h}} \bm{m} + \bm{v} \times \bm{m}\right)
        -\alpha \mathcal{J}(\bm{m}) \times  
           \bm{H}(\bm{m})
   \\
        &\quad 
           - \beta  \mathcal{J}(\bm{m}) \times
      \left( 
           \mathcal{J}(\bm{m}) \times    \bm{H} (\bm{m})
      \right), 
       \end{aligned}
       &\text{in } \Omega \times [0,\infty),
      \\
      \nabla^{\mathfrak{h}}  \bm{m} \cdot\bm{n} = 0,  & \text{on } \partial \Omega \times [0,\infty),  
      \\
    \bm{m}(\cdot,0) = \bm{m}_{0},   &\text{in } \Omega,
    \end{cases}
\end{equation}
where the complex structure approximation $\mathcal{J}(\bm{m})$ is given by $\mathcal{J}(\bm{m}) =\dfrac{\bm{m}}{\max\{|\bm{m}|,1\}}$.

\subsection{Global strong solutions to the approximation system \eqref{II-galerkin1}}\label{ss: Ga-es} We adopt the corresponding Galerkin approximation of the system \eqref{II-galerkin1}:
\begin{equation}\label{II_projection}
    \begin{cases}
       \begin{aligned}
            \partial_{t} \bm{m}_n^\varepsilon - \varepsilon \Delta^{\mathfrak{h}} \bm{m}_n^\varepsilon =&
        -\gamma  \mathcal{P}_n\left(\bm{v} \cdot \nabla^{\mathfrak{h}} \bm{m}_n^\varepsilon +\bm{v} \times \mathcal{J}(\bm{m}_n^\varepsilon)
          \right)
   \\
   & - \alpha \mathcal{P}_n \left(
                   \mathcal{J}(\bm{m}_n^\varepsilon) \times  
         \bm{H}(\bm{m}_n^\varepsilon) 
 \right)
 \\
     &- \beta  \mathcal{P}_n 
     \left( 
     \mathcal{J}(\bm{m}_n^\varepsilon)\times
      \left( 
          \mathcal{J}(\bm{m}_n^\varepsilon)\times  \bm{H}(\bm{m}_n^\varepsilon) 
      \right) \right),
       \end{aligned}
       &\text{in } \Omega \times [0,\infty),
      \\
      \nabla^{\mathfrak{h}}  \bm{m}_n^\varepsilon \cdot\bm{n} = 0,   &\text{on } \partial \Omega \times [0,\infty),  
      \\
      \bm{m}_n^\varepsilon (x,0)= \mathcal{P}_n(\bm{m}_0)(x), &\text{in } \Omega.
    \end{cases}
\end{equation}
By applying similar argument as in Section \ref{ss: Ga-Type-I}, we derive that there exists a regular solution $\bm{m}^\ep_n\in  \mathcal{H}_n$ on a maximal existence time interval $[0, T_n)$. Next, we establish  a priori estimates for this approximate solution $\bm{m}^\ep_n$.

\begin{lemma}
    Suppose $\bm{m}_0 \in L^2(\Omega;\mathbb{S}^2)$, then we have the uniform estimate 
    \begin{align}\label{II-energy}
        \sup\limits_{0\leq t\leq T} \int_{\Omega}
        |\bm{m}_n^\varepsilon|^2 {\rm d}x 
        + 2\varepsilon \int_0^T \int_{\Omega}
        |\nabla^{\mathfrak{h}}\bm{m}_n^\varepsilon|^2 
        {\rm d}x  {\rm d}t
        \leq 
         \int_{\Omega}
        |\bm{m}_0|^2 {\rm d}x,
        \quad \forall T \in (0,T_n).
    \end{align}
This estimate implies that $T_n =+\infty.$
\end{lemma}
\begin{proof}
    
Multiplying the equation in \eqref{II_projection} by $\bm{m}_n^\varepsilon$ and then integrating over $\Omega$ by parts, it yields 
    \begin{align*} 
    \int_{\Omega}
        \partial_t \bm{m}_n^\varepsilon\cdot \bm{m}_n^\varepsilon  {\rm d}x
        - \int_{\Omega}\varepsilon \Delta^{\mathfrak{h}} \bm{m}_n^\varepsilon \cdot \bm{m}_n^\varepsilon  {\rm d}x
&= \gamma\int_{\Omega} 
          \left(\bm{v} \cdot \nabla^{\mathfrak{h}} \bm{m}_n^\varepsilon + \bm{v} \times \mathcal{J}(\bm{m}_n^\varepsilon)
            \right)
      \cdot \bm{m}_n^\varepsilon {\rm d}x
\\
&=\frac{\gamma}{2}\int_{\Omega}  
 \left(\bm{v} \cdot \nabla|\bm{m}_n^\varepsilon|^2\right){\rm d}x
 =0,
\end{align*}
where we have used the facts: $\nabla \cdot \bm{v} =0$ inside $\Om$ and 
$ 
\bm{v}\cdot \bm{n}|_{\partial \Omega \times [0,\infty)}=0.
$
This formula implies the inequality \eqref{II-energy}. 

Therefore, the proof is completed.
\end{proof}

\begin{lemma}\label{II-lemma3.8}
Suppose  $\bm{m}_0 \in H^1(\Omega;\mathbb{S}^2)$, $ \bm{v} $ satisfies the condition \eqref{(a)} or \eqref{(c)}, $\bm{f}\in L^2_{loc}([0,\infty);L^2(\Omega))$, and $\bm{\pi}$ satisfies \eqref{Cpi}. Then, for any $T<\infty$, the solution $\bm{m}_n^\varepsilon$ satisfies
    \begin{equation}\label{II_H^1}
    \begin{split}
&\quad \sup_{t\in[0,T]} \|\bm{m}_n^\varepsilon(t)\|_{H^1_{\mathfrak{h}}(\Omega)}^2
+ \varepsilon \int_0^T \|\Delta^{\mathfrak{h}} \bm{m}_n^\varepsilon\|_{L^2(\Omega)}^2 {\rm d}t 
+ 2\beta \int_0^T\int_{\Omega}  | \mathcal{J}(\bm{m}_n^\varepsilon)\times 
     \Delta^{\mathfrak{h}} \bm{m}_n^\varepsilon |^2{\rm d} x{\rm d}t
\leq \mathfrak{C}_2
\end{split}
\end{equation}
where the constant $\mathfrak{C}_2$ depends only on $\ep, \al,\beta, \gamma, T, \bm{v}, \norm{\bm{f}}_{L^2([0,T]\times \Om)}$ and $\norm{\bm{m}_0}_{H^1(\Om)}$. Moreover, $\mathfrak{C}_2$ is uniform with respect to $\beta$, i.e., $\sup_{0\leq \beta\leq 1}\mathfrak{C}_2<\infty$.
\end{lemma}
\begin{proof}
By choosing  $-\Delta^{\mathfrak{h}} \bm{m}_n^\varepsilon$ as a test function to \eqref{II_projection}, we obtain the following energy estimate:
\begin{equation}\label{II-energy-1}
    \begin{aligned}
\frac{1}{2} \frac{{\rm d}}{{\rm d}t} \|\nabla^{\mathfrak{h}} \bm{m}_n^\varepsilon\|_{L^2(\Omega)}^2
+ 
\varepsilon \|\Delta^{\mathfrak{h}} \bm{m}_n^\varepsilon\|_{L^2(\Omega)}^2 
=&\gamma  \int_{\Omega} 
          \left(
             \bm{v} \cdot \nabla^{\mathfrak{h}} \bm{m}_n^\varepsilon 
            \right)
          \cdot \Delta^{\mathfrak{h}}\bm{m}_n^\varepsilon 
           {\rm d}x+\gamma  \int_{\Omega} 
          \left(\bm{v} \times \mathcal{J}(\bm{m}_n^\varepsilon)\right)
          \cdot \Delta^{\mathfrak{h}}\bm{m}_n^\varepsilon 
          {\rm d}x
\\
&
+\alpha \int_{\Omega}  
\left(\mathcal{J}(\bm{m}_n^\varepsilon)\times   \bm{H}(\bm{m}_n^\varepsilon)\right)\cdot \Delta^{\mathfrak{h}} \bm{m}_n^\varepsilon \:{\rm d} x
\\
&
+\beta\int_\Omega\left(\mathcal{J}(\bm{m}_n^\varepsilon) \times \left(\mathcal{J}(\bm{m}_n^\varepsilon) \times  \bm{H}(\bm{m}_n^\varepsilon)\right)\right)
\cdot \Delta^{\mathfrak{h}} \bm{m}_n^\varepsilon\:{\rm d}x\\
:=&D_1+D_2+D_3+D_4.
\end{aligned}
\end{equation}
The last term $D_4$ can be estimated in the same way as $A_3$ in the proof of Lemma \ref{lemma3.2}. Now, we proceed to estimate the terms $D_1$ to $D_3$. A direct calculation shows 
\begin{align*}
D_2=&\gamma\int_{\Omega} 
(\bm{v} \times \mathcal{J}(\bm{m}_n^\varepsilon) )
\cdot \Delta^{\mathfrak{h}}\bm{m}_n^\varepsilon\mathrm{d} x 
\leq \frac{\varepsilon}{8} 
\|\Delta^{\mathfrak{h}} \bm{m}_n^\varepsilon\|_{L^2(\Omega)}^2 
+ C\frac{\gamma^2}{\varepsilon} \norm{\bm{v}}_{L^3(\Omega)}^2,\\
D_3=&\alpha \int_{\Omega} 
\left(
\mathcal{J}(\bm{m}_n^\varepsilon)\times   \bm{H}(\bm{m}_n^\varepsilon)\right)\cdot \Delta^{\mathfrak{h}} \bm{m}_n^\varepsilon{\rm d} x\\
\leq &\frac{\varepsilon}{8} \|\Delta^{\mathfrak{h}} \bm{m}_n^\varepsilon\|_{L^2(\Omega)}^2
     + \frac{2\alpha^2}{\varepsilon} 
     \left(
     C_{\bm{\pi}}\norm{\bm{m}_n^\varepsilon}_{L^2(\Omega)}^2+\norm{\bm{f} }_{L^2(\Omega)}^2 
     \right).
\end{align*}

Next, we bound the first term $D_1$. For simplicity, we define the $n \times n$ matrix $\mathbb{M}$, whose element in the $i$-th row and $k$-th column is $\partial_i^{\mathfrak{h}} \bm{m}_n^\varepsilon \cdot \partial_k^{\mathfrak{h}}\bm{m}_n^\varepsilon$. Then we obtain
\begin{align*}
	D_1= -\gamma \sum_{i,k=1}^3
	\int_{\Omega} 
	\bm{v}_k 
	\partial_i^{\mathfrak{h}} \partial_k^{\mathfrak{h}} \bm{m}_n^\varepsilon \cdot \partial_i^{\mathfrak{h}} \bm{m}_n^\varepsilon
	{\rm d} x
	-\gamma\int_{\Omega} 
	\nabla \bm{v}\cdot \mathbb{M} 
	{\rm d} x:=\mathfrak{T}_1+ \mathfrak{T}_2.
\end{align*}
Under the condition \eqref{(a)}, $ \nabla \bm{v} $ is anti-symmetric, and thus  $\mathfrak{T}_2=0$ since $ \mathbb{M}$ is symmetric. On the other hand, under the condition \eqref{(c)}, we have
\begin{align*}
\mathfrak{T}_2\leq |\gamma|\norm{\n \bm{v}}_{L^\infty(\Om)}\norm{\n^{\mathfrak{h}} \bm{m}_n^\varepsilon}^2_{L^2(\Om)}.
\end{align*}
For the term $ \mathfrak{T}_1$, we apply the assumptions $\nabla \cdot \bm{v} =0$ and  
$ 
\bm{v}\cdot \bm{n}|_{\partial \Omega \times [0,\infty)}=0
$ to show
\begin{align*}
    \mathfrak{T}_1\overset{\eqref{eq:commutor_estimate}}{=}&
-\gamma
     \int_{\Omega} 
     \sum_{i,k=1}^3
     \left[\bm{v}_k 
     \left(\left(\partial_k^{\mathfrak{h}} \partial_i^{\mathfrak{h}} \bm{m}_n^\varepsilon 
     +\left((\boldsymbol{e}_i \cdot \bm{m}_n^\varepsilon) \boldsymbol{e}_k - (\boldsymbol{e}_k \cdot \bm{m}_n^\varepsilon) \boldsymbol{e}_i \right)\right)\cdot \partial_i^{\mathfrak{h}}\bm{m}_n^\varepsilon\right)
 \right] {\rm d} x
 \\
=&\gamma\int_{\Omega} 
            \nabla \cdot \bm{v} \frac{\left| \nabla^{\mathfrak{h}}\bm{m}_n^\varepsilon \right|^2}{2}{\rm d} x
     - \gamma 
      \int_{\partial \Omega}\bm{v}\cdot \bm{n} \frac{\left| \nabla^{\mathfrak{h}}\bm{m}_n^\varepsilon \right|^2}{2}{\rm d}S
    \\
 & -\gamma
     \int_{\Omega} 
     \sum_{i,k=1}^3 
     \left[ 
            \bm{v}_k\left((\boldsymbol{e}_i \cdot \bm{m}_n^\varepsilon) \boldsymbol{e}_k - (\boldsymbol{e}_k \cdot \bm{m}_n^\varepsilon) \boldsymbol{e}_i \right)\cdot \partial_i^{\mathfrak{h}}\bm{m}_n^\varepsilon
     \right]{\rm d} x
    \\
     \leq& \gamma \| \bm{v}\|_{L^3(\Omega)} 
  \| \nabla^{\mathfrak{h}}\bm{m}_n^\varepsilon\|_{L^2(\Omega)} 
  \| \bm{m}_n^\varepsilon\|_{L^6(\Omega)} 
\\
  \leq &\gamma^2 \| \bm{v}\|^2_{L^3(\Omega)}\| \nabla^{\mathfrak{h}}\bm{m}_n^\varepsilon\|^2_{L^2(\Omega)}+\norm{\bm{m}_0}^2_{L^2(\Om)}
\end{align*}
 Substituting the estimates for $D_1$ to $D_4$ into \eqref{II-energy-1}, we apply Gronwall's inequality to derive the result in \eqref{II_H^1}.
 
 The proof is completed.
\end{proof}

By using the equation \eqref{II_projection}, the estimate \eqref{II_H^1} implies the following bound of $\p_t\boldsymbol{m}_{n}^{\varepsilon}$:
\begin{equation}\label{II-time_derivative}
		\varepsilon\|\partial_{t} \boldsymbol{m}_{n}^{\varepsilon}\|_{L^{2}([0,T]; L^2(\Omega))} \leq \mathfrak{C}_2.
\end{equation}
Therefore, by applying a similar argument as in the proof of Proposition \ref{lemma3.2}, we can show that there exists a subsequence of $\{\boldsymbol{m}_{n}^{\varepsilon}\}$, which converges to a global strong solution $\boldsymbol{m}^\ep$ of \eqref{II-galerkin1}, i.e.,
\begin{equation*}
	\begin{cases}
		\begin{aligned}
			\partial_t \bm{m}^\varepsilon -\varepsilon \Delta^{\mathfrak{h}} \bm{m}^\varepsilon
			&=
			-\gamma \left(\bm{v} \cdot \nabla^{\mathfrak{h}} \bm{m}^\varepsilon + \bm{v} \times \mathcal{J}(\bm{m}^\varepsilon)\right)
			\\
			&\quad 
			-\alpha \mathcal{J}(\bm{m}^\varepsilon) \times  
			\bm{H}(\bm{m}^\varepsilon) 
			- \beta  \mathcal{J}(\bm{m}^\varepsilon) \times
			\left( 
			\mathcal{J}(\bm{m}^\varepsilon) \times    \bm{H}(\bm{m}^\varepsilon) 
			\right),
		\end{aligned}
		& {\rm in } \: \Omega \times [0,\infty),
		\\
		\nabla^{\mathfrak{h}}  \bm{m}^\varepsilon \cdot\bm{n} = 0,  & { \rm on }\:  \partial \Omega \times [0,\infty),  
		\\
		\bm{m}^\varepsilon(\cdot,0)=\bm{m}_{0}, &{\rm in }\: \Omega.
	\end{cases}
\end{equation*}
Moreover, estimates \eqref{II-energy} and \eqref{II-time_derivative} yield the following regularity of $ \bm{m}^\varepsilon$ and $\partial_t \bm{m}^\varepsilon$:
\begin{align*}
    \bm{m}^\varepsilon \in L^\infty([0,T];H^1(\Omega)) \cap L^2([0,T];H^2(\Omega)),
\quad {\rm and }\quad 
\partial_t \bm{m}^\varepsilon \in L^2([0,T];L^2(\Omega))
 \end{align*}
 for any $T<\infty$.

\begin{lemma}\label{II-lemma:max_principle}
For any $T < \infty$, the solution $\bm{m}^\varepsilon$ satisfies $|\bm{m}^\varepsilon| \leq 1$ for almost every $(x,t) \in \Omega \times [0,T]$.
\end{lemma}

\begin{proof}
Following the strategy in the proof of Lemma \ref{I-lemma:max_principle}, we consider the function
\[
\Phi(t) = \frac{1}{2} \int_{\Om} ((|\bm{m}^\varepsilon| - 1)_+)^2 \, {\rm d}x.
\]
Differentiating $\Phi(t)$ with respect to time and substituting the equation for $\partial_t \bm{m}^\varepsilon$ yields
\begin{align*}
    \frac{\mathrm{d} \Phi}{{\rm d}t} 
=\varepsilon \int_{\{|\bm{m}^\varepsilon| > 1\}} 
    \left(\frac{|\bm{m}^\varepsilon| - 1}{|\bm{m}^\varepsilon|}\right) \bm{m}^\varepsilon \cdot\Delta^\mathfrak{h} \bm{m}^\varepsilon 
{\rm d}x
    +  \gamma \int_{\{|\bm{m}^\varepsilon| > 1\}} 
    \left(\frac{|\bm{m}^\varepsilon| - 1}{|\bm{m}^\varepsilon|}\right) \bm{m}^\varepsilon \cdot
    (\bm{v} \cdot \nabla^{\mathfrak{h}} \bm{m}^\varepsilon) {\rm d}x.
\end{align*}
The first term at the right hand side is non-positive, see Lemma \ref{I-lemma:max_principle}. We then analyze the second term:  
\begin{align*}
\int_{\{|\bm{m}^\varepsilon| > 1\}} 
    \left(\frac{|\bm{m}^\varepsilon| - 1}{|\bm{m}^\varepsilon|}\right) \bm{m}^\varepsilon \cdot
    (\bm{v} \cdot \nabla^{\mathfrak{h}} \bm{m}_n^\varepsilon)    {\rm d}x
&= 
    \int_\Omega 
         \left(\frac{|\bm{m}^\varepsilon| - 1}{|\bm{m}^\varepsilon|}\right)_+ 
         \bm{v} \cdot \nabla \frac{| \bm{m}^\varepsilon |^2}{2} {\rm d}x
\\
&= 
    \int_\Omega  (|\bm{m}^\varepsilon| - 1 )_+
     \bm{v} \cdot \nabla | \bm{m}^\varepsilon | {\rm d}x
\\
&=
    \frac{1}{2}\int_\Omega 
     \bm{v} \cdot \nabla ((|\bm{m}^\varepsilon| - 1 )_+)^2 {\rm d}x=0.
\end{align*}

Therefore, $\Phi(t)$ is a decreasing non-negative function. Since $\Phi(0)=0$, we have $\Phi (t) \equiv 0, \forall t >0$, and hence $|\bm{m}^\varepsilon|\leq 1$ a.e.  in $\Omega \times [0,T]$. The proof is completed.  
\end{proof}

Lemma \ref{II-lemma:max_principle} implies $\mathcal{J}(\bm{m}^\varepsilon)  = \bm{m}^\varepsilon$. Consequently, $\bm{m}^\varepsilon$ solves the following  problem 
\begin{equation}\label{II-maximum1}
\begin{cases}
\begin{aligned}
\partial_t \bm{m}^\varepsilon -\varepsilon \Delta^{\mathfrak{h}} \bm{m}^\varepsilon
&=
    -\gamma \left(\bm{v} \cdot \nabla^{\mathfrak{h}} \bm{m}^\varepsilon +\bm{v} \times \bm{m}^\varepsilon\right)
           \\
      &\quad 
 -\alpha \bm{m}^\varepsilon \times  
           \bm{H}(\bm{m}^\varepsilon) 
           - \beta  \bm{m}^\varepsilon\times
      \left( 
           \bm{m}^\varepsilon \times    \bm{H}(\bm{m}^\varepsilon) 
      \right),
\end{aligned}
& {\rm in } \: \Omega \times [0,\infty),
      \\
      \nabla^{\mathfrak{h}}  \bm{m}^\varepsilon \cdot\bm{n} = 0,  & { \rm on }\:  \partial \Omega \times [0,\infty),  
      \\
    \bm{m}^\varepsilon(\cdot,0)=\bm{m}_{0}, &{\rm in }\: \Omega.
\end{cases}
\end{equation}

\medskip
\subsection{Global type-II weak solutions to \eqref{system1} with $\beta>0$}
In this part, we establish a priori uniform estimates for $\bm{m}^\varepsilon$  with respect to $\varepsilon$.  
\begin{lemma}\label{II-uniform_energy_estimate}
Suppose that $\bm{m}_0 \in H^1(\Omega;\mathbb{S}^2)$, $ \bm{v}$ satisfies  the condition \eqref{(a)}, $\bm{f}\in L^2_{loc}([0,\infty);L^2(\Omega))$ and $\bm{\pi}$ satisfies \eqref{Cpi}.
For any $T >0$,  $\bm{m}^\varepsilon$ admits the  estimates
\begin{align}
\sup\limits_{t \in (0,T)} \| \bm{m}^\varepsilon(t)\|_{H_{\mathfrak{h}}^1(\Omega)}^2
+ \beta \int_0^T\int_{\Omega}  
| \bm{m}^\varepsilon\times 
     \Delta^{\mathfrak{h}} \bm{m}^\varepsilon |^2{\rm d} x{\rm d}t
\leq 
 \mathfrak{C}_3,\label{II-3.12}
 \end{align}
 and 
 \begin{align}
 \|\partial_t \bm{m}^\varepsilon\|_{L^2(0,T;H^{-1}(\Omega))} \leq &\mathfrak{C}_3,\label{II-3.13}
 \end{align}
 where the constant $\mathfrak{C}_3$ depends only on $\al,\beta, \gamma, T, \bm{v}, \norm{\bm{f}}_{L^2([0,T]\times \Om)}$ and $\norm{\bm{m}_0}_{H^1(\Om)}$.   
\end{lemma}
\begin{proof}
 By multiplying the equation in \eqref{II-maximum1} by $-\Delta^{\mathfrak{h}} \bm{m}^\varepsilon$ and integrating over $\Omega$, we apply similar argument as in the proof of Lemma \ref{lemma3.2} to obtain 
\begin{equation}\label{II-energy-2}
    \begin{aligned}
&\frac{1}{2} \frac{{\rm d}}{{\rm d}t} \|\nabla^{\mathfrak{h}} \bm{m}^\varepsilon\|_{L^2(\Omega)}^2
+\frac{\beta}{2} \int_{\Omega}  |\bm{m}^\varepsilon\times 
     \Delta^{\mathfrak{h}} \bm{m}^\varepsilon |^2{\rm d} x\\
\leq& C\gamma\left( \| \bm{v}\|_{L^3(\Omega)}^2 
+
\|\nabla^{\mathfrak{h}}\bm{m}_n^\varepsilon\|_{L^2(\Omega)}^2 \right) 
     +\left(\beta+ \frac{2\alpha^2}{\beta}\right)\int_{\Omega} |\bm{\pi}( \bm{m}^\varepsilon)|^2 +|\bm{f}|^2 {\rm d}x
\end{aligned}
\end{equation}
which yields \eqref{II-3.12} by Gronwall's inequality. 

Next, our aim is to establish the bound \eqref{II-3.13}. For any test function $ \boldsymbol{\phi} \in L^{2}([0,T]; H^{1}_0(\Omega))$, we derive from the estimate \eqref{II-3.12}  and $\norm{\boldsymbol{m}^{\varepsilon}}_{L^\infty}\leq 1$ that
\begin{align*}
&\int_{0}^{T} \int_{\Omega} \partial_{t} \boldsymbol{m}^{\varepsilon} \cdot \boldsymbol{\phi} \, \mathrm{d}x \, \mathrm{d}t \\
=&-\gamma\int_{0}^{T} \int_{\Omega} 
       \left( \boldsymbol{v} \cdot \nabla^{\mathfrak{h}} \boldsymbol{m}^{\varepsilon} 
       + \boldsymbol{v} \times \bm{m}^\varepsilon
       \right) \cdot \boldsymbol{\phi} \, \mathrm{d}x \, \mathrm{d}t -\varepsilon \int_{0}^{T} \int_{\Omega} \n^{\mathfrak{h}} \boldsymbol{m}^{\varepsilon} \cdot \n^{\mathfrak{h}}\boldsymbol{\phi} \, \mathrm{d}x \, \mathrm{d}t \\ 
&-\int_{0}^{T} \int_{\Omega} (\alpha \bm{m}^\varepsilon \times \Delta^{\mathfrak{h}} \boldsymbol{m}^{\varepsilon} \cdot \boldsymbol{\phi} )\:\mathrm{d}x \, \mathrm{d}t
    -\int_{0}^{T} \int_{\Omega} \beta(\bm{m}^\varepsilon \times (\bm{m}^\varepsilon \times \Delta^{\mathfrak{h}} \boldsymbol{m}^{\varepsilon})) \cdot \boldsymbol{\phi} \, \mathrm{d}x \, \mathrm{d}t\\
& -\int_{0}^{T} \int_{\Omega} (\alpha \bm{m}^\varepsilon \times (\boldsymbol{\pi}(\boldsymbol{m}^{\varepsilon}) + \boldsymbol{f})) \cdot \boldsymbol{\phi}\mathrm{d}x \, \mathrm{d}t -\int_{0}^{T} \int_{\Omega} \beta (\bm{m}^\varepsilon \times (\bm{m}^\varepsilon \times (\boldsymbol{\pi}(\boldsymbol{m}^{\varepsilon}) + \boldsymbol{f}))) \cdot \boldsymbol{\phi} \, \mathrm{d}x \, \mathrm{d}t\\
\leq& \mathfrak{C}_3\|\boldsymbol{\phi}\|_{L^{2}([0,T]; H^{1}_{0}(\Omega))}.
\end{align*}
As a consequence, by taking the supremum over all $ \boldsymbol{\phi} $ with $ \|\boldsymbol{\phi}\|_{L^{2}([0,T]; H^{1}_{0}(\Omega))} = 1 $, we obtain the estimate \eqref{II-3.13}. 

Therefore, the proof is completed.
\end{proof}

Now, we are in the position to establish Theorem \ref{theroem1'}.
\begin{proof}[\textbf{Proof of Theorem \ref{theroem1'}}]
	
The proof is divided into two steps.

\medskip
\noindent\textbf{Step 1: The convergence of $\bm{m}^\varepsilon$.} For any $T < \infty$,  Lemma \ref{II-uniform_energy_estimate} implies that
\[
\|\bm{m}^\varepsilon\|_{L^{\infty}([0,T];H^{1}(\Omega))} + \|\partial_t \bm{m}^\varepsilon\|_{L^{2}([0,T];H^{-1}(\Omega))} \leq C,
\]
for some constant $C$ independent of $\varepsilon$.
Then, there exists  $\bm{m} \in L^{\infty}([0,T];L^{2}(\Omega))$ such that
\begin{align*}
\bm{m}^\varepsilon &\rightharpoonup \bm{m}, \quad \text{weakly* in } L^{\infty}([0,T];H^{1}(\Omega)), \text{ as } \varepsilon \to 0, \\
\nabla^\mathfrak{h}  \bm{m}^\varepsilon &\rightharpoonup \nabla^\mathfrak{h}  \bm{m}, \quad \text{weakly in } L^{2}([0,T];L^{2}(\Omega)), \text{ as } \varepsilon \to 0.
\end{align*}
Let $\mathcal{X} = H^{1}(\Omega)$, $\mathcal{B} =L^{2}(\Omega)$ and  $\mathcal{Y} =  H^{-1}(\Omega)$. Lemma \ref{LAS} implies 
\[
\bm{m}^\varepsilon \to \bm{m}, \quad \text{strongly in } C^{0}([0,T];L^p(\Omega)), \text{ as } \varepsilon \to 0,
\]
for $2\leq p<6$. Since $\bm{v}\in L^2([0,T];L^3(\Om))$ and $|\bm{m}^\varepsilon|\leq 1$, we derive from the above convergence results of $\bm{m}^\varepsilon$ that
\begin{align*}
\bm{v} \cdot \nabla^{\mathfrak{h}} \bm{m}^\varepsilon
+ \bm{v} \times  \bm{m}^\varepsilon&\rightharpoonup\bm{v} \cdot \nabla^{\mathfrak{h}} \bm{m}
+ \bm{v} \times  \bm{m},\quad \text{weakly in } L^{1}([0,T];L^{1}(\Omega)).
\end{align*}

Additionally, due to  
\[
\int_{\Omega} |\bm{m}^\varepsilon|^{2}{\rm d}x + 2\varepsilon \int_{0}^{T} \int_{\Omega} |\nabla^{\mathfrak{h}} \bm{m}^\varepsilon|^{2}{\rm d}x{\rm d}t = \int_{\Omega} |\bm{m}_0|^{2}{\rm d}x,
\]
we take $\varepsilon \to 0$ to obtain
\[
\int_{\Omega} (|\bm{m}|^{2} - 1){\rm d}x = 0.
\]
This implies $|\bm{m}| = 1$ for a.e. $(x,t) \in \Omega \times [0,T]$, since $|\bm{m}|\leq 1$.

\medskip
\noindent\textbf{Step 2: Global existence of type-II weak solution to \eqref{system1}.} For any $\bm{\varphi} \in C^{\infty}(\bar{\Omega} \times [0,T])$, the solution $\bm{m}^\varepsilon$ satisfies
\begin{align}\label{identity-1}
\begin{aligned}
&\quad 
    \int_{0}^{T} \int_{\Omega} \left( \partial_t \bm{m}^\varepsilon\cdot \bm{\varphi} \right) {\rm d}x{\rm d}t  
+\varepsilon \int_{0}^{T} \int_{\Omega} \left( \n^\mathfrak{h} \bm{m}^\varepsilon\cdot \n^\mathfrak{h}\bm{\varphi} \right) {\rm d}x{\rm d}t
\\
&= - \gamma \int_{0}^{T} \int_{\Omega} 
     (\bm{v} \cdot \nabla^{\mathfrak{h}} \bm{m}^\varepsilon + \bm{v} \times  \bm{m}^\varepsilon)
           \cdot \bm{\varphi} {\rm d}x{\rm d}t
           - \alpha \int_{0}^{T} \int_{\Omega} \left( \bm{m}^\varepsilon \times (\bm{\pi}(\bm{m}^\varepsilon)+ \bm{f}) \right)\cdot \bm{\varphi}  {\rm d}x{\rm d}t 
\\
&\quad +\alpha \int_{0}^{T} \int_{\Omega} \left( \bm{m}^\varepsilon \times \n^\mathfrak{h} \bm{m}^\varepsilon\right)\cdot \n^\mathfrak{h}\bm{\varphi}  {\rm d}x{\rm d}t 
- \beta \int_{0}^{T} \int_{\Omega} \left( \bm{m}^\varepsilon \times ( \bm{m}^\varepsilon \times \Delta^\mathfrak{h} \bm{m}^\varepsilon)\right)\cdot \bm{\varphi} {\rm d}x{\rm d}t
\\
&\quad -\beta \int_{0}^{T} \int_{\Omega} \left( \bm{m}^\varepsilon \times ( \bm{m}^\varepsilon \times (\bm{\pi}(\bm{m}^\varepsilon)+ \bm{f}) )\right)\cdot \bm{\varphi} {\rm d}x{\rm d}t.
\end{aligned}
\end{align}

Since $\bm{m}^\varepsilon \to \bm{m}$ strongly in $C^{0}([0,T];L^{2}(\Omega)), \: \bm{m}^\varepsilon \to \bm{m},\: \text{a.e. } (x,t) \in \Omega \times [0,T]$, the left-hand side (LHS) of \eqref{identity-1} satisfies 
\begin{align*}
\text{LHS of } \eqref{identity-1} \to \int_{\Omega} \left( \bm{m}\cdot \bm{\varphi} \right) (T){\rm d}x - \int_{\Omega} \left( \bm{m}_0\cdot \bm{\varphi} \right) (0) {\rm d}x- \int_{0}^{T} \int_{\Omega} \left( \bm{m}\cdot \partial_t \bm{\varphi} \right) {\rm d}x{\rm d}t \quad{\rm as} \:\:\ep\to 0.
\end{align*}

In the following,  we analyze the convergence of terms at the right hand side of \eqref{identity-1}. Using the convergence results for $\bm{m}^\varepsilon$ in Step 1 and then letting $  \varepsilon \to 0$, we obtain 
\begin{align*}
\gamma\int_{0}^{T} \int_{\Omega} 
	(\bm{v} \cdot \nabla^{\mathfrak{h}} \bm{m}^\varepsilon
	+ \bm{v} \times  \bm{m}^\varepsilon)\cdot \bm{\varphi} {\rm d}x{\rm d}t
	\to &
	\:\gamma\int_{0}^{T} \int_{\Omega} 
	(\bm{v} \cdot \nabla^{\mathfrak{h}} \bm{m}  + \bm{v} \times  \bm{m} )\cdot \bm{\varphi} {\rm d}x{\rm d}t,\\
\int_{0}^{T} \int_{\Omega} \left( \bm{m}^\varepsilon \times \n^\mathfrak{h} \bm{m}^\varepsilon\right)\cdot \n^\mathfrak{h}\bm{\varphi}  {\rm d}x{\rm d}t 
\to &\int_{0}^{T} \int_{\Omega} \left( \bm{m} \times \nabla^\mathfrak{h} \bm{m}\right) \cdot \nabla^\mathfrak{h} \bm{\varphi} {\rm d}x{\rm d}t,
\\
     \int_{0}^{T} \int_{\Omega} \left( \bm{m}^\varepsilon \times (\bm{\pi}(\bm{m}^\varepsilon)+ \bm{f}) \right)\cdot \bm{\varphi} {\rm d}x{\rm d}t
\to &  \int_{0}^{T} \int_{\Omega} \left( \bm{m} \times (\bm{\pi}(\bm{m})+ \bm{f}) \right)\cdot \bm{\varphi} {\rm d}x{\rm d}t,
\\
     \int_{0}^{T} \int_{\Omega} \left( \bm{m}^\varepsilon \times ( \bm{m}^\varepsilon \times (\bm{\pi}(\bm{m}^\varepsilon)+ \bm{f}) )\right)\cdot \bm{\varphi} {\rm d}x{\rm d}t
\to & \int_{0}^{T} \int_{\Omega} \left( \bm{m} \times ( \bm{m} \times (\bm{\pi}(\bm{m})+ \bm{f}) )\right)\cdot \bm{\varphi} {\rm d}x{\rm d}t.
\end{align*} 
It remains to show the convergence of the following term:
\begin{align*}
    &\int_{0}^{T} \int_{\Omega} \left(\bm{m}^\varepsilon\times ( \bm{m}^\varepsilon \times \Delta^\mathfrak{h} \bm{m}^\varepsilon)\right) \cdot \bm{\varphi} {\rm d}x{\rm d}t\\ 
 = &\int_{0}^{T} \int_{\Omega} 
        \n^\mathfrak{h} \left( 
            \bm{m} \times \bm{\varphi}\right)
            \cdot
             (\bm{m}^\varepsilon \times \n^\mathfrak{h} \bm{m}^\varepsilon)
             {\rm d}x{\rm d}t 
+ 
             \int_{0}^{T} \int_{\Omega} 
        \left( 
            (\bm{m}- \bm{m}^\varepsilon )\times \bm{\varphi}\right)
            \cdot
             (\bm{m}^\varepsilon \times \Delta^\mathfrak{h} \bm{m}^\varepsilon)
             {\rm d}x{\rm d}t 
    \\
    \to&  \int_{0}^{T} \int_{\Omega} 
    \n^\mathfrak{h} \left( 
    \bm{m} \times \bm{\varphi}\right)
    \cdot
    (\bm{m} \times \n^\mathfrak{h} \bm{m})
    {\rm d}x{\rm d}t 
    \\
=& \int_{0}^{T} \int_{\Omega}  
\nabla^\mathfrak{h} \bm{m} \cdot \nabla^\mathfrak{h} \bm{\varphi}
-  |\nabla^\mathfrak{h}\bm{m}|^2(\bm{\varphi} \cdot  \bm{m}) 
           {\rm d}x{\rm d}t.
\end{align*}
Here we have applied the following estimate
\begin{align*}
\left|\int_{0}^{T} \int_{\Omega} 
\left( 
(\bm{m}- \bm{m}^\varepsilon )\times \bm{\varphi}\right)
\cdot
(\bm{m}^\varepsilon \times \Delta^\mathfrak{h} \bm{m}^\varepsilon)
{\rm d}x{\rm d}t\right|\leq \mathfrak{C}_3\norm{\bm{\varphi}}_{L^\infty}\norm{\bm{m}- \bm{m}^\varepsilon}_{L^2}\to 0 \quad {\rm as}\:\: \ep\to 0.
\end{align*}

Therefore,  we conclude that the limiting map $\bm{m}$ is a type-II weak solution to \eqref{system1}, namely
\begin{align*}
&\quad \int_{\Omega} \left( \bm{m}\cdot \bm{\varphi} \right) (T){\rm d}x - \int_{\Omega} \left( \bm{m}_0\cdot \bm{\varphi} \right)(0) {\rm d}x - \int_{0}^{T} \int_{\Omega} \left( \bm{m}\cdot \partial_t \bm{\varphi} \right) {\rm d}x{\rm d}t \\
=&  -\gamma\int_{0}^{T} \int_{\Omega} 
    \left(
             \bm{v} \cdot \nabla^{\mathfrak{h}} \bm{m}  +\bm{v} \times  \bm{m} 
            \right)\cdot \bm{\varphi} {\rm d}x{\rm d}t+\alpha\int_{0}^{T} \int_{\Omega} \left( \bm{m} \times \nabla^\mathfrak{h}  \bm{m}\cdot \nabla^\mathfrak{h}  \bm{\varphi} \right) {\rm d}x{\rm d}t\\
 &- \beta\int_{0}^{T} \int_{\Omega} \nabla^\mathfrak{h}  \bm{m} \cdot \nabla^{\mathfrak{h}}\bm{\varphi}{\rm d}x{\rm d}t
            + \beta\int_{0}^{T} \int_{\Omega} |\nabla^\mathfrak{h}\bm{m}|^2 \bm{m} \cdot \bm{\varphi}{\rm d}x{\rm d}t
\\
&-\int_{0}^{T} \int_{\Omega}  \left( \alpha\bm{m} \times (\bm{\pi}(\bm{m})+ \bm{f}) +  \beta\left( \bm{m} \times ( \bm{m} \times (\bm{\pi}(\bm{m})+ \bm{f}) )\right)\right)\cdot \bm{\varphi} {\rm d}x{\rm d}t
\end{align*}
for any $\bm{\varphi} \in C^{\infty}(\bar{\Omega} \times [0,T])$ and any $T < \infty$.

The proof is completed. 
\end{proof}

\subsection{Global weak solutions to incompressiable Schr\"odinger flow }\label{sec:inf}
In the last subsection, we prove Theorem \ref{theorem2}. We consider the following approximate system of \eqref{eq:IS}, namely the system \eqref{II-galerkin1} with $\beta=0$:
\begin{equation}\label{II-galerkin2}
	\begin{cases}
\partial_{t} \bm{m} - \varepsilon\Delta^{\mathfrak{h}} \bm{m}
			=-\gamma\left(\bm{v} \cdot \nabla^{\mathfrak{h}} \bm{m} + \bm{v} \times \bm{m}\right)
			-\alpha \mathcal{J}(\bm{m}) \times  
			\bm{H}(\bm{m}),
		&\text{in } \Omega \times [0,\infty),
		\\
		\nabla^{\mathfrak{h}}  \bm{m} \cdot\bm{n} = 0,  & \text{on } \partial \Omega \times [0,\infty),  
		\\
		\bm{m}(\cdot,0) = \bm{m}_{0},   &\text{in } \Omega.
	\end{cases}
\end{equation}
Applying the Galerkin approximation scheme introduced in section~\ref{ss: Ga-es} with $\beta=0$, we obtain the existence of a global solution $\bm{m}^\varepsilon$ to \eqref{II-galerkin2} satisfying
\begin{align*}
	\bm{m}^\varepsilon \in L^\infty([0,T];H^1(\Omega)) \cap L^2([0,T];H^2(\Omega)), \quad 
	\partial_t \bm{m}^\varepsilon \in L^2([0,T];L^2(\Omega)),
\end{align*}
for any $T<\infty$, together with the uniform bound \[\norm{\bm{m}^\varepsilon}_{L^\infty(\Om\times[0,T])}\leq 1.\]

Next, we  derive uniform estimates for $\bm{m}^\varepsilon$ with respect to the parameter $\varepsilon$.
\begin{lemma}
Suppose that $\bm{m}_0 \in H^1(\Omega;\mathbb{S}^2)$, $ \bm{v} $ satisfies the condition \eqref{(b)} or \eqref{(c)}, $\bm{f}\in L^2_{loc}([0,\infty);L^2(\Omega))$ and $\bm{\pi}$ satisfies the bound \eqref{Cpi1}.
Then for any $T<\infty$, the approximate solution $\bm{m}^\varepsilon$ admits the following uniform estimates:
\begin{align}
 \sup_{0 \leq t \leq T} \left(\|\bm{m}^\varepsilon\|_{L^2(\Omega)}^2+ 
    \|\nabla^{\mathfrak{h}}\bm{m}^\varepsilon\|_{L^2(\Omega)}^2
    \right)
 \leq &
    \mathfrak{C}_4,\label{III-energy-4}\\
    \|\partial_t \bm{m}^\varepsilon\|_{L^2([0,T];H^{-1}(\Omega))} 
    \leq  &  \mathfrak{C}_4,\label{III-4.11}
 \end{align}
where the constant $\mathfrak{C}_4$ is independent of $\ep$.
\end{lemma}
\begin{proof} 
Without loss of generality, we establish the estimates under the condition \eqref{(b)}: 
$ 	\bm{v} \in  L^2_{loc}([0,\infty); H^1(\Omega)), \: \bm{v}\cdot \bm{n}|_{\partial \Omega \times  [0, \infty)}=0, \:\nabla \bm{v } =-( \nabla \bm{v })^T\,\,\text{in} \,\, \Omega\times [0, \infty),
$
since the same bounds can be obtained under the condition \eqref{(c)}. 

 Testing the equation in \eqref{II-galerkin2} by $-\Delta^{\mathfrak{h}}\bm{m}^\varepsilon$ and integrating over $\Omega$, we obtain
\begin{equation}\label{II-energy-3}
	\begin{aligned}
		&\frac{1}{2} \frac{{\rm d}}{{\rm d}t} \|\nabla^{\mathfrak{h}} \bm{m}^\varepsilon\|_{L^2(\Omega)}^2+\ep\norm{\De^{\mathfrak{h}} \bm{m}^\varepsilon}^2_{L^2(\Omega)}=E_1+E_2+E_3.
	\end{aligned}
\end{equation}
Here the terms $E_1$ to $E_3$ can be estimated as follows
\begin{align*}
  E_1:=&\gamma\int_{\Omega} 
     (\bm{v} \times \bm{m}^\varepsilon)
      \cdot \Delta^{\mathfrak{h}}\bm{m}^\varepsilon\mathrm{d} x \\
    =& - \gamma\int_{\Omega} 
     ( \nabla^{\mathfrak{h}} \bm{v} \times \bm{m}^\varepsilon) \cdot \nabla^{\mathfrak{h}} \bm{m}^\varepsilon \mathrm{d} x \leq 2|\gamma|\left( \|\bm{v}\|_{H^1(\Omega)}^2 +\|\nabla^{\mathfrak{h}}\bm{m}^\varepsilon\|_{L^2(\Omega)}^2 \right),\\
 E_2:=&\gamma  \int_{\Omega}  
          \left(
           \bm{v} \cdot \nabla^{\mathfrak{h}} \bm{m}^\varepsilon 
           \right)
          \cdot \Delta^{\mathfrak{h}}\bm{m}^\varepsilon 
          {\rm d}x\\
 =& -\gamma
     \int_{\Omega} 
     \sum_{i,k=1}^3
     \left[\bm{v}_k 
     \left(
         \partial_i^{\mathfrak{h}} \partial_k^{\mathfrak{h}} \bm{m}^\varepsilon \cdot \partial_i^{\mathfrak{h}} \bm{m}^\varepsilon
      \right) 
       \right] {\rm d} x
      -\gamma\int_{\Omega} 
     \sum_{i,k=1}^3
     \left[
      \partial_i \bm{v}_k \left(\partial_k^{\mathfrak{h}} \bm{m}^\varepsilon\cdot \partial_i^{\mathfrak{h}} \bm{m}^\varepsilon \right) 
 \right] {\rm d} x
 \\
=& -\gamma
     \int_{\Omega} 
     \sum_{i,k=1}^3 
            \left((\boldsymbol{e}_i \cdot \bm{m}^\varepsilon) \boldsymbol{e}_k - (\boldsymbol{e}_k \cdot \bm{m}^\varepsilon) \boldsymbol{e}_i \right)
     \left(     \bm{v}_k\cdot \partial_i^{\mathfrak{h}}\bm{m}^\varepsilon
     \right){\rm d} x\\
 \leq& \frac{\gamma}{2} 
    \left( \|\bm{v}\|_{H^1(\Omega)}^2 +\|\nabla^{\mathfrak{h}}\bm{m}^\varepsilon\|_{L^2(\Omega)}^2 \right),
\end{align*}
where we have applied the condition \eqref{(b)} in the above estimate of $E_2$, and
\begin{align*}
E_3:=&\alpha \int_{\Omega}  
\left(\bm{m}^\varepsilon\times   \bm{H}(\bm{m}^\varepsilon)\right)\cdot \Delta^{\mathfrak{h}} \bm{m}^\varepsilon {\rm d} x\\
=&
  -\alpha \int_{\Omega}  
\left(
\nabla^{\mathfrak{h}}\bm{m}^\varepsilon\times  (\bm{\pi}(\bm{m}^\varepsilon)+ \bm{f})\right)
+\bm{m}^\varepsilon\times   \nabla^{\mathfrak{h}}(\bm{\pi}(\bm{m}^\varepsilon)+ \bm{f})
\cdot \nabla^{\mathfrak{h}} \bm{m}^\varepsilon {\rm d} x
\\
\leq&\alpha
\left(\|\nabla^{\mathfrak{h}}\bm{m}^\varepsilon\|_{L^2(\Omega)}^2 
+ C_{\bm{\pi}} \|\bm{m}^\varepsilon\|_{H^1_{\mathfrak{h}}(\Omega)}^2 
+ \|\bm{f}\|_{H^{1}(\Omega)}^2 
\right).
\end{align*}
Combining these estimates in \eqref{II-energy-3} and applying Gronwall’s inequality yields \eqref{III-energy-4}.

The bound \eqref{III-4.11} follows by an argument analogous to that used in the proof of \eqref{II-3.13}. This completes the proof. 
\end{proof}

\begin{proof}[\textbf{The proof of Theorem~\ref{theorem2}}]
With the uniform estimates \eqref{III-energy-4} and \eqref{III-4.11}, we follow the same compactness and limiting arguments as in the proof of Theorem~\ref{theroem1'}, along with $\beta=0$, to conclude the proof of Theorem~\ref{theorem2}. Here, we omit the details.
\end{proof}
\appendix
\section{Equivalent estimates of Sobolev norms with helical derivatives}\label{appendix}
This appendix provides detailed proofs of the properties of the Sobolev space associated with helical derivatives, namely Propositions~\ref{proposition2.1}–\ref{proposition2}, which are stated in Section~\ref{sec:preliminaries}.
\begin{proof}[\textbf{The proof of Proposition~~\ref{proposition2.1}}]
For any $\boldsymbol{u}\in H^2(\Om)$, a direct computation yields
    \begin{align*}
   \int_{\Omega} |\nabla^{\mathfrak{h}} \boldsymbol{u}|^{2} \mathrm{d} x&= \sum_{i=1}^3\int_{\Omega}|\p_i\boldsymbol{u}-\bm{e}_i\times  \boldsymbol{u}|^{2} \mathrm{d} x
         \leq C \norm{ \bm{u} }^2_{H^1(\Omega)}.
    \end{align*}
   Consequently, we also have
    \begin{align*}
         \int_{\Omega} |\nabla\boldsymbol{u}|^{2} \mathrm{d} x
         \leq C  \int_{\Omega} \left( |\nabla^{\mathfrak{h}} \boldsymbol{u}|^2 + | \boldsymbol{u}|^2\right)\mathrm{d} x 
         \leq  C\norm{ \bm{u} }_{H^1_{\mathfrak{h}}(\Omega)}^2
    \end{align*}
which implies the estimate \eqref{equiv-1} with $k=1$.
    
 We next consider the case of $k=2$. By the definition of the covariant derivatives $\partial_i^{\mathfrak{h}}$, we compute
    \begin{align*}
        \sum_{i,j=1}^{3}  \int_{\Omega} \left|  \partial_{j}^\mathfrak{h} \partial_{i}^\mathfrak{h} \boldsymbol{u} \right|^2 \mathrm{d} x
       =
         \sum_{i,j=1}^{3}  \int_{\Omega} \left|\partial_{ji}\boldsymbol{u} - \bm{e}_i \times \partial_j\bm{u} - \bm{e}_j \times \partial_i \bm{u} + \bm{e}_j \times (\bm{e}_i  \times \bm{u}) \right|^2 \mathrm{d} x
        \leq  C  \norm{ \bm{u} }^2_{H^2(\Omega)}.
    \end{align*}
  Conversely, by a similar argument, we obtain 
    \begin{align*}
         \sum_{i,j=1}^{3}  \int_{\Omega} \left|  \partial_{j}\partial_{i}\boldsymbol{u} \right|^2 \mathrm{d} x
         \leq 
         C  \norm{ \bm{u} }^2_{H_\mathfrak{h} ^2(\Omega)}.
    \end{align*}
    
 Combining the above estimates, we conclude that the norms induced by $\nabla$ and $\nabla^{\mathfrak{h}}$ are equivalent up to order two. The reverse inequality can be proved in a similar manner, and we omit the details.
\end{proof}

In the end of this appendix, we give the proof of Proposition \ref{proposition2}.
\begin{proof}[\textbf{The proof of Proposition \ref{proposition2}}.]
We first observe that the helical derivative acting on a vector field $\bm{u}$ can be written as
\begin{equation}\label{A.2.1}
    \partial_i^{\mathfrak{h}} \bm{u} = \partial_i \bm{u} - M_i \bm{u},
\end{equation}
where $M_i, \:i=1,2,3,$ is skew-symmetric matrix given explicitly by
\begin{align*}
M_1=\begin{pmatrix}
0& 0 & 0 \\
 0&0  &  -1\\
 0& 1 & 0
\end{pmatrix},\quad
M_2=\begin{pmatrix}
0 &0  &  1\\
 0&0  & 0\\
 -1 & 0 & 0
\end{pmatrix},\quad
M_3=\begin{pmatrix}
 0& -1 &0  \\
 1&  0& 0\\
  0& 0 & 0
\end{pmatrix}.
\end{align*}

As a consequence, the helical Laplacian $\Delta^{\mathfrak{h}}$ admits the following expansion:
\begin{align}\label{eq:helical_laplacian_expanded}
\Delta^{\mathfrak{h}} \bm{u} 
=\Delta \bm{u} - 2\sum_{i=1}^3 M_i \partial_i \bm{u} - \bm{u}.
\end{align}
Next, we rewrite the chiral boundary condition $\nabla^{\mathfrak{h}} \bm{u} \cdot \bm{n} = 0$, where $\bm{n} = (\bm{n}_1, \bm{n}_2, \bm{n}_3)$ denotes the unit outward normal vector of $\p\Om$. Using \eqref{A.2.1}, this condition is equivalent to
\begin{equation}
\nabla \bm{u} \cdot \bm{n}=\bm{n}_i M_i\bm{u} \quad \text{on } \partial\Omega. \label{eq:bc_rewritten}
\end{equation}

Let $\bm{u}\in H^{2}(\Omega)$ satisfies the chiral boundary condition
$\nabla^{\mathfrak{h}} \bm{u} \cdot \bm{n} = 0$.
By the Agmon-Douglis-Nirenberg regularity theorem
(see \cite{MR125307} or \cite[Theorem~3.2]{MR2030823}),
there exists a constant $C>0$ such that
\begin{align}\label{ADs}
	\norm{\bm{u}}_{H^2(\Omega)} \leq 
	C \left( \norm{\Delta \bm{u}}_{L^2 (\Omega)}  
	+\norm{ \nabla \bm{u} \cdot \bm{n}}_{H_{\partial}^{1}(\Omega)} 
	+\norm{\bm{u}}_{H^1(\Omega)}\right),
\end{align}
where the boundary Sobolev norm is defined by
\begin{align*}
\| \nabla\bm{u} \cdot \bm{n}\|_{H_{\partial}^{1}(\Omega) }:= \inf \bigl\{\|G\|_{H^{1}(\Omega)} \mid G \in H^{1}(\Omega), G|_{\partial \Omega} =  \nabla \bm{u} \cdot \bm{n}\bigr\}.
\end{align*}
Using \eqref{eq:helical_laplacian_expanded} and \eqref{eq:bc_rewritten}, we have the following estimates: 
\begin{align*}
\norm{\Delta \bm{u}}_{L^2 (\Omega)}  
    &\leq 
   C \left( 
   \norm{\Delta^\mathfrak{h} \bm{u}}_{L^2 (\Omega)} 
+ \norm{\bm{u}}_{H^1(\Omega)}
   \right),
\\  
 \norm{ \nabla \bm{u} \cdot \bm{n}}_{H_{\partial}^{1}(\Omega) }
    &= \norm{\bm{n} \times \bm{u}}_{H_{\partial}^{1}(\Omega) }
   \leq C \norm{\bm{u}}_{H^{1}(\Omega)}.
   \end{align*}
Therefore, for any $\bm{u}\in H^2(\Om)$ with the chiral boundary condition $\nabla^{\mathfrak{h}} \bm{u} \cdot \bm{n} = 0$, we derive from \eqref{ADs} that 
\begin{align*}
\norm{\bm{u}}_{H^2{(\Omega)}}\leq& C 
	\left(
	\norm{\Delta^{\mathfrak{h}}\bm{u}}_{L^2(\Omega)} + 
	\norm{\bm{u}}_{H^1(\Omega)}
	\right).
\end{align*}

Finally, by using  Sobolev interpolation  and Young's inequality, we control the $H^{1}$-norm as follows:
   \begin{align*}
    \norm{\bm{u}}_{H^{1}(\Omega)} 
    \leq C  \norm{\bm{u}}_{H^{2}(\Omega)}^\frac{1}{2}  \norm{\bm{u}}_{L^{2}(\Omega)}^\frac{1}{2}
    \leq  \frac{1}{2} \norm{\bm{u}}_{H^{2}(\Omega)} +C \norm{\bm{u}}_{L^{2}(\Omega)},  
\end{align*}
which  yields
\[\norm{\bm{u}}_{H^2{(\Omega)}}\leq C 
\left(
\norm{\Delta^{\mathfrak{h}}\bm{u}}_{L^2(\Omega)} + 
\norm{\bm{u}}_{L^2(\Omega)}\right).\]

Therefore, the proof is completed. 
\end{proof}
%
%
%

\section*{Acknowledgments}
The author B. Chen is supported partially by NSFC (Grant No. 12301074) and Guangdong Basic and Applied Basic Research Foundation (Grant No. 2025 A1515010502), and the author Z. Qiu is supported by NSFC No. 12171109 and No. 12201140, and the Innovation Research for the Postgraduate of Guangzhou University, No. JCCX2024-013.

\medskip
\section*{Statements and Declarations}
\noindent {\it\bf{Conflict of interest}}: The authors declare that they have no conflict of interest.

\medskip
\noindent {\it\bf{Data Availability}}: Data sharing is not applicable to this article as no datasets were generated or analyzed during the study.



\end{sloppypar}
\end{document}